\documentclass[11pt,reqno]{amsart}
\usepackage{amssymb,mathrsfs,graphicx,subfigure, enumerate}
\usepackage{amsmath,amsfonts,amssymb,amscd,amsthm,bbm}
\usepackage{extpfeil}
\usepackage{graphicx,colortbl}
\usepackage{float}
\usepackage{epsfig}
\usepackage{caption}
\usepackage{bm}
\graphicspath{{Figure/}}
\usepackage{kotex}
\usepackage[margin=1in]{geometry}

\makeatletter
\@namedef{subjclassname@2020}{\textup{2020} Mathematics Subject Classification}
\makeatother

\title[Navier-Stokes-Korteweg equations]{Time-Asymptotic Stability of the Stationary Solution to the Impermeable Wall Problem for the radially Symmetric Navier–Stokes–Korteweg Equations}

\author[Kim]{Jeongho Kim}
\address[Jeongho Kim]{\newline Department of Applied Mathematics, \newline Kyung Hee University, 1732 Deogyeong-daero, Giheung-gu, Yongin-si, Gyeonggi-do 17104, Republic of Korea}
\email{jeonghokim@khu.ac.kr}

\begin{document}
\newtheorem{theorem}{Theorem}[section]
\newtheorem{lemma}{Lemma}[section]
\newtheorem{corollary}{Corollary}[section]
\newtheorem{proposition}{Proposition}[section]
\newtheorem{remark}{Remark}[section]
\newtheorem{definition}{Definition}[section]
\renewcommand{\theequation}{\thesection.\arabic{equation}}
\renewcommand{\thetheorem}{\thesection.\arabic{theorem}}
\renewcommand{\thelemma}{\thesection.\arabic{lemma}}
\newcommand{\bbr}{\mathbb R}
\newcommand{\R}{\mathbb{R}}
\newcommand{\e}{\varepsilon}
\newcommand{\pa}{\partial}
\newcommand{\<}{\left\langle}
\renewcommand{\>}{\right\rangle}


\begin{abstract}
	We study the asymptotic behavior of the initial-boundary value problem for the radially symmetric Navier--Stokes--Korteweg (NSK) equations defined on the exterior domain $\Omega = \{x\in\R^n~|~|x|> 1\}$. In particular, we consider the impermeable wall problem, where the velocity at the boundary $\{x\in\R^n~|~|x|=1\}$ is set to be zero. We show that, if the initial data is a small perturbation of the stationary solution, and the boundary data are sufficiently small, then there exists a global-in-time strong solution to the radially symmetric NSK equations, and it converges to the stationary solution time-asymptotically. Our method is based on elementary energy estimates with a combination of carefully designed energy functionals.
\end{abstract}

\maketitle

\section{Introduction}\label{sec:1}
\setcounter{equation}{0}

The motion of compressible fluid with capillary effect is described by the so-called Navier--Stokes--Korteweg (NSK) system, which has been used to understand two-phase fluids \cite{AMH98} and quantum viscous fluids \cite{AS22,BGL19,J10}. To present the exact governing equations, let $\Omega\subset \R^n$ be the domain occupied by the fluid, and $\rho(t,x)$ and $U(t,x)$ be density and velocity of the fluid for $t\ge0$ and $x\in\Omega$. Then, the NSK system reads as

\begin{align}
	\begin{aligned}\label{eq:NSK}
	&\rho_t +\mbox{div}(\rho U)=0,\quad t>0,\quad x\in \Omega,\\
	&\rho (U_t +U\cdot\nabla U)+\nabla P(\rho)=\nu\Delta U+(\nu+\lambda)\nabla(\mbox{div}U)+\kappa\rho\nabla\Delta \rho.
	\end{aligned}
\end{align}

Here, $P(\rho)=\rho^\gamma$ is a barotropic pressure, $\nu$ and $\lambda$ are viscosity coefficients, and $\kappa$ is a capillarity coefficient. The idea of introducing capillarity effect on the fluid dates back to van der Waals \cite{W94} and Korteweg \cite{K01}, where a prototype of the capillary fluid model was considered. Later, the NSK system \eqref{eq:NSK} was rigorously justified in \cite{DS85} by using the thermodynamics of interstitial working. After the NSK model was introduced, numerous results such as well-posedness \cite{CH13,DD01,H11,HL96} and stability \cite{C12,CHZ15,FLP26,FPZ23,HKKL25} were studied in the last decades. For the readers who are interested in the recent results on the NSK system, we refer to \cite{AS22,BGL19,GL16,HK26} and references therein for a detailed, though not exhaustive, list of works on it.

In this paper, we focus on the NSK equations under the radially symmetric assumption. To be specific, we consider the domain $\Omega$ as an exterior domain $\Omega:=\{x\in \R^n~|~ |x|>1\}$, and consider the density and velocity that are given in the radially symmetric form: 
\begin{equation}\label{eq:spherically_symmetric}
	\rho(t,x)=\rho(t,r),\quad U(t,x)=\frac{x}{r}u(t,r),\quad r:=|x|.
\end{equation}
Then, after substituting \eqref{eq:spherically_symmetric} into \eqref{eq:NSK}, one finds that $(\rho,u)$ satisfies the following radially symmetric NSK equations:

\begin{align}
	\begin{aligned}\label{eq:NSK_spherical}
	&\rho_t +\frac{(r^{n-1}\rho u)_r}{r^{n-1}}=0,\quad t>0,\quad r>1,\\
	&\rho(u_t+uu_r)+P(\rho)_r =\mu\left(\frac{(r^{n-1}u)_r}{r^{n-1}}\right)_r +\kappa\rho\left(\rho_{rr}+\frac{n-1}{r}\rho_r\right)_r,
	\end{aligned}
\end{align}
subject to the initial data $(\rho_0,u_0)$, where $\mu :=2\nu +\lambda$ and $\rho_0(r)>0$. At $r=1$ and at $r=+\infty$, we impose the boundary conditions and the far-field condition given as
\begin{equation}\label{eq:boundary}
	u(t,1) = u_-,\quad \rho_r(t,1)=\rho_b,\quad\lim_{r\to+\infty}(\rho(t,r),u(t,r))=(\rho_+,0),\quad t\ge0,
\end{equation}
where $\rho_+>0$, $u_-$, and $\rho_b$ are constants. We also impose the following compatibility condition:
\begin{align}\label{eq:compatibility}
\left[\rho_0u_0u_{0r}+P(\rho_0)_r -\mu\left(\frac{(r^{n-1}u_0)_r}{r^{n-1}}\right)_r-\kappa\rho_0\left(\rho_{0rr}+\frac{n-1}{r}\rho_{0r}\right)_r\right]_{r=1}=0.
\end{align}

When $u_-=0$, the initial boundary value problem \eqref{eq:NSK_spherical}--\eqref{eq:boundary} is called an {\it impermeable wall problem}, while when $u_->0$ or $u_-<0$, an {\it inflow problem} or an {\it outflow problem} respectively. In our previous work \cite{Kpre}, we proved that under the smallness condition on the boundary data, there exists a unique stationary solution to each problem, and investigated their properties such as decay rates or vanishing capillarity limit.

In the present work, we focus on the impermeable wall problem ($u_-=0$) and our goal is to show the time-asymptotic behavior of the solution towards the unique stationary solution $(\tilde{\rho},0)$, where $\tilde{\rho}$ is a solution to the stationary equation
\begin{equation}\label{eq:stationary}
	P(\tilde{\rho})_r = \kappa\tilde{\rho}\left(\tilde{\rho}_{rr}+\frac{n-1}{r}\tilde{\rho}_r\right)_r,\quad \tilde{\rho}_r(1)=\rho_b,\quad \lim_{r\to+\infty}\tilde{\rho}(r)=\rho_+.
\end{equation}

\begin{theorem}\label{thm:main}
	Let $n\ge 3$ and $(\tilde{\rho},0)$ be the stationary solution to \eqref{eq:NSK_spherical} satisfying \eqref{eq:stationary}. Then, there exist positive constants $\delta_0$ and $\e_0$ such that the following holds. Suppose that the boundary value satisfies $|\rho_b|<\delta_0$ and the initial data $(\rho_0-\tilde{\rho},u_0)\in H^3(1,\infty)\times H^2(1,\infty)$ satisfy
	\[\|(\rho_0-\tilde{\rho},u_0)\|_{H^2(1,\infty)\times H^1(1,\infty)}<\e_0,\]
	the boundary conditions
	\[ u_0(1) = 0, \quad \rho_{0r}(1) = \rho_b,\]
	and the compatibility condition \eqref{eq:compatibility}. Then, there exists a unique global-in-time solution $(\rho,u)$ to \eqref{eq:NSK_spherical} satisfying
	\[(\rho-\tilde{\rho},u)\in C([0,\infty);H^3(1,\infty))\times C([0,\infty);H^2(1,\infty)).\]
	Furthermore, the solution asymptotically converges to the stationary solution:
	\[\lim_{t\to\infty}(\|\rho(t)-\tilde{\rho}\|_{W^{1,\infty}(1,\infty)}+\|u(t)\|_{L^\infty(1,\infty)})=0.\]
\end{theorem}

\begin{remark}
	Here, the constraint on the dimension $n\ge3$ is essential in our analysis, since this assumption is crucially used to obtain functional inequalities in Lemma \ref{lem:Hardy} and Lemma \ref{lem:exp-Hardy}. 
\end{remark}

\subsection*{Related works}
Time-asymptotic behavior of compressible fluid has been studied over the past decades. Although there is an extensive literature, we will focus on the results regarding radially symmetric models. For the history and recent progress in the large-time behavior of the one-dimensional compressible Navier--Stokes (NS) equations on the whole space $\R$ or the half space $\R_+:=\{x\in\R~|~x>0\}$, we refer to the recent papers \cite{HKKKO26,KVW25} and references therein.

Time-asymptotic behavior of the radially symmetric NS equations, that is when $\kappa=0$ in \eqref{eq:NSK_spherical}, was first investigated in \cite{J96}, where the well-posedness and time-asymptotic behavior for the impermeable wall problem were proved when $n=3$. This result is generalized in \cite{NNY04} to the large initial data case, in the presence of external potential force. The result was later extended to the radially symmetric Navier--Stokes--Fourier (NSF) equations \cite{NN08}. When the velocity at the boundary is non-zero, i.e., when the inflow/outflow problems are considered, the existence of the stationary solution and the convergence towards it are rather new subjects. The existence of the stationary solution to the inflow and outflow problems of NS equations is proved in \cite{HM21}, and the time-asymptotic behaviors of the outflow problem and inflow problem were investigated in \cite{HNS24,HNpre} and \cite{HHN26}, respectively. For the case of a viscous heat-conducting fluid, \cite{HMpre} proved the existence of a stationary solution to the radially symmetric NSF equations.

On the other hand, the time-asymptotic behavior of compressible viscous-capillary fluid has also been studied recently. For the one-dimensional NSK equations on the whole space $\R$, the convergence towards the rarefaction waves \cite{C12}, viscous-dispersive shock wave \cite{CHZ15,HKKL25} and their composite wave \cite{HK26} were proved. Time-asymptotic behavior of initial boundary value problems of the one-dimensional NSK equations is also studied in \cite{HKOpre,LTY22,LXC23,LZ21}. However, to the best of the author's knowledge, time-asymptotic behavior of the radially symmetric NSK equations \eqref{eq:NSK_spherical} has not been studied. Recently, the author proved in \cite{Kpre} that there exists a unique stationary solution to the impermeable wall, inflow, and outflow problems of \eqref{eq:NSK_spherical}, under the smallness assumption on the boundary data, and studied the vanishing capillarity limit of the stationary solution satisfying \eqref{eq:stationary}. As a follow-up, the aim of the present work is to prove the time-asymptotic behavior of the impermeable wall problem to the radially symmetric NSK equations towards its stationary solution. \\

The rest of the paper is organized as follows. In Section \ref{sec:2}, we provide several preliminaries that will be used throughout the paper. Then, we conduct the energy estimates on the perturbation between the solution $(\rho,u)$ and the stationary solution $(\tilde{\rho},0)$ in Section \ref{sec:perturbation_estimate}, which provides an {\it a priori} estimate on the $H^2\times H^1$ norm of the perturbation. Finally, in Section \ref{sec:global}, we combine the local existence theory and the {\it a priori} estimate obtained in Section \ref{sec:perturbation_estimate} to derive the global well-posedness by using a continuation argument, and as a result, we prove the desired time-asymptotic behavior in Theorem \ref{thm:main}.\\

\noindent {\bf Notation}. For any function $f:[1,\infty)\to \R$, we introduce the weighted $L^2$-norm and $L^\infty$-norm on $(1,\infty)$ as
\[\|f\|^2:=\|f\|^2_{L^2_{r^{n-1}}(1,\infty)}:=\int_1^\infty |f(r)|^2r^{n-1}\,dr,\quad \|f\|_{L^\infty}:=\sup_{r\ge 1}|f(r)|.\]
Accordingly, we denote the $L^2$-inner product between $f,g:[1,\infty)\to\R$ as 
\[\<f,g\>:=\int_1^\infty f(r)g(r)r^{n-1}\,dr.\] 
Furthermore, the Sobolev norms are also defined as
\[\|f\|_{H^k(1,\infty)}^2=\|f\|_{H^k}^2:=\sum_{i=0}^k\|\pa_r^{i} f\|^2=\int_1^\infty \left(\sum_{i=0}^k|\pa_r^i f(r)|^2\right)r^{n-1}\,dr.\]
Finally, we also use the notation
\[\|f\|^2_{L^2_{r^{n-m}}}:=\int_1^\infty |f(r)|^2r^{n-m}\,dr\]
for the $L^2$-norm with the other weight $r^{n-m}$.

\section{Preliminaries}\label{sec:2}
\setcounter{equation}{0}
In this section, we present the preliminaries, such as several functional inequalities and the existence and properties of the stationary solution $\tilde{\rho}$. We first collect some identities and inequalities on the functions defined on $[1,\infty)$ which will be used in the later estimates.

\begin{lemma}\label{lem:separate}
	For any function $f:[1,\infty)\to\R$ satisfying $f(1)=\lim_{r\to\infty}f(r)=0$, we have
	\[\int_1^\infty\left(f_r+\frac{n-1}{r}f\right)^2r^{n-1}\,dr = \|f_r\|^2+(n-1)\|f\|^2_{L^2_{r^{n-3}}}.\]
\end{lemma}
\begin{proof}
	By direct expansion and using integration by parts, we get
	\begin{align*}
		\int_1^\infty \left(f_r+\frac{n-1}{r}f\right)^2r^{n-1}\,dr &= \|f_r\|^2+ (n-1)^2\|f\|_{L^2_{r^{n-3}}}^2+2(n-1)\int_1^\infty ff_rr^{n-2}\,dr\\
		&=\|f_r\|^2+ (n-1)^2\|f\|_{L^2_{r^{n-3}}}^2-(n-1)(n-2)\int_1^\infty f^2r^{n-3}\,dr\\
		&=\|f_r\|^2+ (n-1)\|f\|^2_{L^2_{r^{n-3}}}.
	\end{align*}
\end{proof}

\begin{lemma}\label{lem:Hardy}
	If $f(1)=\lim_{r\to\infty}f(r)=0$, then $\|f\|_{L^2_{r^{n-3}}}\le \frac{2}{n-2}\|f_r\|$. Moreover, if $\lim_{r\to\infty}f(r)=0$, we have $\|f\|_{L^\infty}^2\le \frac{1}{n-2}\|f_r\|^2$.
\end{lemma}
\begin{proof}
	Since $\pa_r(r^{n-2})=(n-2)r^{n-3}$, we use integration by parts and boundary condition to get
	\[(n-2)\int_1^\infty f^2 r^{n-3}\,dr = \int_1^\infty f^2\pa_r(r^{n-2})\,dr=-2\int_1^\infty ff_r r^{n-2}\,dr.\]
	Then, using the Cauchy--Schwarz inequality and $r\ge1$, we get
	\[(n-2)\|f\|^2_{L^2_{r^{n-3}}}\le 2\|f\|_{L^2_{r^{n-3}}}\|f_r\|_{L^2_{r^{n-1}}}=2\|f\|_{L^2_{r^{n-3}}}\|f_r\|,\]
	which yields the first desired inequality. Next, for any $r\ge 1$, we again use the Cauchy--Schwarz inequality to get
	\begin{align*}
		|f(r)|^2=\left|\int_r^\infty f_s\,ds\right|^2\le \int_r^\infty s^{-(n-1)}\,ds\int_r^\infty f_s^2s^{n-1}\,ds\le \frac{\|f_r\|^2}{(n-2)r^{n-2}}\le \frac{\|f_r\|^2}{n-2}.
	\end{align*}
	Taking supremum over $r\ge1$, we get the second inequality.
\end{proof}

\begin{lemma}\label{lem:exp-Hardy}
	If $\lim_{r\to\infty}f(r)=0$ and $\sigma>0$, we have
	\[\int_1^\infty f^2 e^{-2\sigma r}r^{n-1}\,dr\le C(n,\sigma)\|f_r\|^2.\]
\end{lemma}
\begin{proof}
	It follows from the previous lemma that for any $r\ge 1$, $|f(r)|^2\le \frac{1}{(n-2)r^{n-2}}\|f_r\|^2$. Therefore,
	\[\int_1^\infty f^2 e^{-2\sigma r}r^{n-1}\,dr\le \frac{\|f_r\|^2}{n-2}\int_1^\infty r e^{-2\sigma r}\,dr=\frac{1}{n-2}\left(\frac{1}{2\sigma}+\frac{1}{4\sigma^2}\right)e^{-2\sigma}\|f_r\|^2.\]
\end{proof}

Now, we recall the existence and decay properties of the stationary solution $\tilde{\rho}$.

\begin{proposition}\cite{Kpre}
	For sufficiently small $|\rho_b|$, there exists a unique smooth solution $\tilde{\rho}(r)$ to \eqref{eq:stationary} satisfying
	\[|\tilde{\rho}(r)-\rho_+|\le C_k|\rho_b|e^{-\sigma r},\quad |\tilde{\rho}^{(k)}(r)|\le C_k|\rho_b|e^{-\sigma r},\quad k\ge1,\quad r\ge 1\]
	for some $\sigma>0$.
\end{proposition}
In \cite{Kpre}, only the convergence of $|\tilde{\rho}(r)-\rho_+|$ is stated. However, the decay of the derivatives $\tilde{\rho}^{(k)}$ follows easily from the equation \eqref{eq:stationary}. 

\section{{\it A priori} estimate for the perturbation}\label{sec:perturbation_estimate}
\setcounter{equation}{0}

In this section, we present the {\it a priori} estimate for the perturbation of the solution $(\rho,u)$ to the radially symmetric NSK equations \eqref{eq:NSK_spherical} around the stationary solution $(\tilde{\rho},0)$ satisfying \eqref{eq:stationary}. To this end, we define the perturbations of the density and velocity as
\[\phi(t,r):=\rho(t,r)-\tilde{\rho}(r),\quad \psi(t,r):=u(t,r).\]
For simplicity, we introduce the following linear operators:
\[L(f):=f_{rr}+\frac{n-1}{r}f_r-\frac{n-1}{r^2}f=\left(\frac{(r^{n-1}f)_r}{r^{n-1}}\right)_r,\quad A(f):=f_{rr}+\frac{n-1}{r}f_r = \frac{(r^{n-1}f_r)_r}{r^{n-1}}.\]
Then, since $\tilde{\rho}_t = 0$, it follows from \eqref{eq:NSK_spherical}$_1$ that $\phi$ satisfies
\[\phi_t + \frac{(r^{n-1}\rho\psi)_r}{r^{n-1}}=0.\]
On the other hand, it follows from \eqref{eq:NSK_spherical}$_2$ that $\psi$ satisfies
\begin{equation}\label{eq:psi_1}
\rho\psi_t +\rho\psi\psi_r +P(\rho)_r = \mu L(\psi)+\kappa\rho A(\rho)_r.
\end{equation}
Using the equation \eqref{eq:stationary}, which can be written as $P(\tilde{\rho})_r = \kappa\tilde{\rho}A(\tilde{\rho})_r$, \eqref{eq:psi_1} can be modified as
\begin{equation*}
	\begin{aligned}\label{eq:psi_2}
	\rho\psi_t+\rho\psi\psi_r + (P(\rho)-P(\tilde{\rho}))_r &=\mu L(\psi) + \kappa\rho A(\rho)_r-\kappa\tilde{\rho}A(\tilde{\rho})_r\\
	&=\mu L(\psi) +\kappa\rho A(\phi)_r+\kappa\phi A(\tilde{\rho})_r\\
	&=\mu L(\psi) +\kappa\rho A(\phi)_r+h'(\tilde{\rho})\tilde{\rho}_r\phi.
	\end{aligned}
\end{equation*}
Here, we used 
\[\kappa A(\tilde{\rho})_r = \frac{P(\tilde{\rho})_r}{\tilde{\rho}}=h'(\tilde{\rho})\tilde{\rho}_r,\quad h(\rho):=\int^\rho \frac{P'(s)}{s}\,ds.\]
Thus, the perturbation $(\phi,\psi)$ satisfies the following equations:
\begin{align}
	\begin{aligned}\label{eq:perturbation}
		&\phi_t + \frac{(r^{n-1}\rho \psi)_r}{r^{n-1}}=0,\\
		&\rho\psi_t +\rho\psi\psi_r +(P(\rho)-P(\tilde{\rho}))_r = \mu L(\psi)+\kappa\rho A(\phi)_r+h'(\tilde{\rho})\tilde{\rho}_r\phi,
	\end{aligned}
\end{align}
subject to the boundary conditions
\begin{equation*}\label{boundary_perturbation}
	\psi(t,1) = 0,\quad \phi_r(t,1)=0,\quad \lim_{r\to+\infty}(\phi,\psi)(t,r)=(0,0),\quad t>0.
\end{equation*}

In the following, we present the energy estimate on the perturbations $(\phi,\psi)$. Let $N$ be a $H^2\times H^1$-norm of $(\phi,\psi)$:
\[N^2(t):=\|\phi(t)\|_{H^2}^2+\|\psi(t)\|_{H^1}^2=\|\phi\|^2+\|\phi_r\|^2+\|\phi_{rr}\|^2+\|\psi\|^2+\|\psi_r\|^2,\]
and $\mathcal{D}$ be the dissipation defined as
\[\mathcal{D}(t):=\|G\|^2+\|G_r\|^2+\|\phi_r\|^2+\|\phi_{rr}\|^2+\|\psi_t\|^2,\] 
where $G$ is 
\begin{equation}\label{eq:G}
	G:=\frac{(r^{n-1}\psi)_r}{r^{n-1}}=\psi_r + \frac{n-1}{r}\psi,\quad \mbox{and therefore}\quad L(\psi) = G_r.
\end{equation}
Furthermore, from Lemma \ref{lem:separate} and the boundary condition $\psi(1)=0$, it directly follows that
\begin{align*}
	\|G\|^2 =\|\psi_r\|^2 + (n-1)\|\psi\|_{L^2_{r^{n-3}}}^2.
\end{align*}
Similar decomposition also holds for $A(\phi)$: since $\phi_r(1)=0$,
\[\|A(\phi)\|^2 =\|\phi_{rr}\|^2+(n-1)\|\phi_r\|_{L^2_{r^{n-3}}}^2.\]
The goal of this section is to attain the following {\it a priori} estimate.

\begin{proposition}\label{prop:apriori}
	There exist small enough positive constants $\delta$ and $\e$ such that the following holds. Let $(\rho,u)$ be a solution to \eqref{eq:NSK_spherical} on $[0,T]$ satisfying the boundary condition \eqref{eq:boundary} with $|\rho_b|<\delta$, and let $\tilde{\rho}$ be a stationary solution satisfying \eqref{eq:stationary}. Suppose that the following a priori smallness condition holds:
	\begin{equation}\label{smallness_assumption}
		\sup_{0\le t\le T}N(t)<\e.
	\end{equation}
	Then, there exists a positive constant $C_0$ that is independent of $\delta$, $\e$, and $T$ such that 
	\begin{equation}\label{est:apriori}
		\sup_{0\le t\le T} N^2(t) + \int_0^T\mathcal{D}(s)\,ds \le C_0N^2(0).
	\end{equation}
\end{proposition}

The {\it a priori} smallness assumption \eqref{smallness_assumption} implies that the $L^\infty$-norms of the perturbations are also small, that is
\begin{equation*}\label{smallness_Linf}
	\sup_{0\le t\le T}\|(\phi,\psi)\|_{L^\infty}<C\e.
\end{equation*}
Since the stationary solution $\tilde{\rho}$ has positive lower and upper bounds, \eqref{smallness_assumption} implies that there exist $0<\rho_m<\rho_M$ such that
\[\rho_m<\rho(t,r),\tilde{\rho}(r)<\rho_M,\quad\mbox{for all}\quad 0<t<T,\quad r\ge 1.\]

Let us briefly outline the proof of Proposition \ref{prop:apriori}. In the following, we will obtain four estimates on the different energy functionals in Lemma \ref{lem:L2}--Lemma \ref{lem:Phi}. These energy functionals are combined to construct another energy functional $\mathcal{L}$ defined in \eqref{def:L}, which is shown to be equivalent to $N^2$.

We start with the $L^2$-estimate on the perturbations $(\phi,\psi)$ to control $N(t)$.

\subsection{$L^2$-estimate}
To obtain $L^2$-estimate, we define the internal energy and its relative quantity as
\[Q(\rho):=\int^\rho h(\eta)\,d\eta,\quad \mbox{and}\quad Q(\rho|\tilde{\rho}):=Q(\rho)-Q(\tilde{\rho})-Q'(\tilde{\rho})(\rho-\tilde{\rho}).\]
In particular, if $P(\rho) = \rho^\gamma$ with $\gamma \ge1$, then
\[Q(\rho)=\begin{cases}
	\frac{\rho^\gamma}{\gamma-1}\quad &\mbox{if}\quad \gamma>1,\\
	\rho\log\rho-\rho\quad&\mbox{if}\quad \gamma=1,
\end{cases}\quad\mbox{and}\quad Q(\rho|\tilde{\rho})=\begin{cases}
	\frac{1}{\gamma-1}(\rho^\gamma-\tilde{\rho}^\gamma-\gamma\tilde{\rho}^{\gamma-1}(\rho-\tilde{\rho})),\quad&\mbox{if}\quad \gamma>1,\\
	\rho\log\frac{\rho}{\tilde{\rho}}-(\rho-\tilde{\rho}),\quad&\mbox{if}\quad \gamma=1.
\end{cases}\]

\begin{lemma}\label{lem:L2}
	Under the same assumptions as in Proposition \ref{prop:apriori}, we have
	\[\frac{d}{dt}\int_1^\infty \left(\frac{1}{2}\rho\psi^2 + Q(\rho|\tilde{\rho})+\frac{\kappa}{2}\phi_r^2\right)r^{n-1}\,dr+\mu\|G\|^2=0.\]
\end{lemma}
\begin{proof}
	We multiply \eqref{eq:perturbation}$_2$ by $\psi$ and integrate over $(1,\infty)$ to get
	\begin{equation}
		\begin{aligned}\label{est:L2-1}
		\int_1^\infty &\left(\rho\psi\psi_t +\rho\psi^2\psi_r+(P(\rho)-P(\tilde{\rho}))_r\psi \right)r^{n-1}\,dr\\
		&=\int_1^\infty\left(\mu L(\psi)\psi+\kappa\rho A(\phi)_r\psi+h'(\tilde{\rho})\tilde{\rho}_r\phi\psi\right)r^{n-1}\,dr.
		\end{aligned}
	\end{equation}
	We first consider the left-hand side of \eqref{est:L2-1}. Using \eqref{eq:NSK_spherical}$_1$ and integration by parts, we get
	\begin{align*}
		\int_1^\infty \rho\psi\psi_t r^{n-1}\,dr &= \frac{1}{2}\frac{d}{dt}\int_1^\infty \rho\psi^2 r^{n-1}\,dr -\frac{1}{2}\int_1^\infty \rho_t \psi^2 r^{n-1}\,dr\\
		&=\frac{1}{2}\frac{d}{dt}\int_1^\infty \rho\psi^2 r^{n-1}\,dr +\frac{1}{2}\int_1^\infty (r^{n-1}\rho\psi)_r \psi^2 \,dr\\
		&=\frac{1}{2}\frac{d}{dt}\int_1^\infty \rho\psi^2 r^{n-1}\,dr -\int_1^\infty \rho\psi^2 \psi_r r^{n-1}\,dr.
	\end{align*}
	Thus, the first two terms of the left-hand side of \eqref{est:L2-1} become
	\[\int_1^\infty (\rho\psi\psi_t+\rho\psi^2\psi_r)r^{n-1}\,dr = \frac{1}{2}\frac{d}{dt}\int_1^\infty \rho\psi^2 r^{n-1}\,dr.\]
	Also, we use $P'(\rho) = \rho h'(\rho)$ to get
	\begin{align*}
		\int_1^\infty (P(\rho)-P(\tilde{\rho}))_r\psi r^{n-1}\,dr=\int_1^\infty \rho h'(\rho)\rho_r \psi r^{n-1}\,dr-\int_1^\infty \tilde{\rho}h'(\tilde{\rho})\tilde{\rho}_r \psi r^{n-1}\,dr.
	\end{align*}
	Next, we estimate the right-hand side of \eqref{est:L2-1}. Noticing that $L(\psi)=G_r$ and using \eqref{eq:G}, the first term of the right-hand side of \eqref{est:L2-1} becomes
	\[\mu\int_1^\infty L(\psi)\psi r^{n-1}\,dr = -\mu\int_1^\infty G(r^{n-1}\psi)_r\,dr=-\mu\int_1^\infty G^2 r^{n-1}\,dr=-\mu\|G\|^2,\]
	On the other hand, the second term of the right-hand side of \eqref{est:L2-1} can be reformulated by using \eqref{eq:perturbation}$_1$ as
	\begin{align*}
		\int_1^\infty \rho A(\phi)_r\psi r^{n-1}\,dr&=-\int_1^\infty A(\phi)(\rho\psi r^{n-1})_r\,dr = \int_1^\infty A(\phi)\phi_t r^{n-1}\,dr\\
		&=\int_1^\infty \phi_t (r^{n-1}\phi_r)_r\,dr=-\int_1^\infty \phi_{tr}\phi_r r^{n-1}\,dr=-\frac{1}{2}\frac{d}{dt}\int_1^\infty \phi_r^2 r^{n-1}\,dr.
	\end{align*}	
	Combining all the above computations, we get
	\begin{align*}
		\frac{d}{dt}&\int_1^\infty\left(\frac{1}{2}\rho\psi^2+\frac{\kappa}{2}\phi_r^2\right)r^{n-1}\,dr+\mu\|G\|^2+\int_1^\infty \rho h'(\rho)\rho_r\psi r^{n-1}\,dr-\int_1^\infty \rho h'(\tilde{\rho})\tilde{\rho}_r\psi r^{n-1}\,dr=0.
	\end{align*}
	Finally, observing
	\begin{align*}
		\frac{d}{dt}\int_1^\infty Q(\rho|\tilde{\rho})r^{n-1}\,dr&=\int_1^\infty \left(Q'(\rho)\rho_t-Q'(\tilde{\rho})\tilde{\rho}_t -Q''(\tilde{\rho})\phi\tilde{\rho}_t -Q'(\tilde{\rho})(\rho-\tilde{\rho})_t\right)r^{n-1}\,dr\\
		&=\int_1^\infty (Q'(\rho)-Q'(\tilde{\rho}))\rho_t r^{n-1}\,dr = -\int_1^\infty  (h(\rho)-h(\tilde{\rho}))(r^{n-1}\rho\psi)_r\,dr\\
		&=\int_1^\infty (h'(\rho)\rho_r-h'(\tilde{\rho})\tilde{\rho}_r)\rho\psi r^{n-1}\,dr,
	\end{align*}
	we have the desired estimate.
\end{proof}

Since $\rho$ and $\tilde{\rho}$ are bounded below and above, the relative internal energy $Q(\rho|\tilde{\rho})$ is equivalent to $|\rho-\tilde{\rho}|^2=|\phi|^2$: there exists a positive constant $C$ that only depends on $\rho_m$ and $\rho_M$ such that
\[C^{-1}|\rho-\tilde{\rho}|^2\le Q(\rho|\tilde{\rho})\le C|\rho-\tilde{\rho}|^2.\]
Hence, the quantity $\int_1^\infty Q(\rho|\tilde{\rho})r^{n-1}\,dr$ is equivalent to the $L^2$-norm of $\phi$. Therefore, Lemma \ref{lem:L2} implies that we have a control on $H^1$-norm of $\phi=\rho-\tilde{\rho}$ and $L^2$-norm of $\psi=u$ and the dissipation $\|G\|^2$. Next, we derive the estimate on the cross term between $\psi$ and $\phi_r$, which provides us further dissipation terms $\|\phi_r\|^2$ and $\|\phi_{rr}\|^2$.

\begin{lemma}\label{lem:cross}
	Under the same assumptions as in Proposition \ref{prop:apriori}, there exist positive constants $c_2$ and $C_2$ such that 
	\[\frac{d}{dt}\int_1^\infty \rho\psi\phi_r r^{n-1}\,dr +c_2(\|\phi_r\|^2+\|\phi_{rr}\|^2)\le C_2\|G\|^2+C(\delta+\e^2)\mathcal{D}.\]
\end{lemma}
\begin{proof}
	We multiply \eqref{eq:perturbation}$_2$ by $\phi_r r^{n-1}$ and integrate over $(1,\infty)$ to get
	\begin{align}
		\begin{aligned}\label{est:cross-00}
		\int_1^\infty&\left(\rho\psi_t \phi_r +\rho\psi\psi_r\phi_r +(P(\rho)-P(\tilde{\rho}))_r\phi_r\right)r^{n-1}\,dr\\
		&=\mu\int_1^\infty L(\psi)\phi_rr^{n-1}\,dr +\kappa\int_1^\infty\rho A(\phi)_r\phi_rr^{n-1}\,dr +\int_1^\infty h'(\tilde{\rho})\tilde{\rho}_r\phi\phi_r r^{n-1}\,dr.
		\end{aligned}
	\end{align}
	Since $r^{n-1}\rho_t =r^{n-1}\phi_t= -(r^{n-1}\rho \psi)_r$, we use integration by parts to obtain
	\begin{align}
		\begin{aligned}\label{est:cross-01}
		\int_1^\infty\rho\psi_t\phi_r r^{n-1}\,dr &= \frac{d}{dt}\int_1^\infty \rho\psi\phi_r r^{n-1}\,dr-\int_1^\infty \rho_t\psi\phi_r r^{n-1}\,dr -\int_1^\infty \rho\psi\phi_{tr}r^{n-1}\,dr\\
		&=\frac{d}{dt}\int_1^\infty \rho\psi\phi_r r^{n-1}\,dr -\int_1^\infty \rho \psi(\psi_r\phi_r+\psi\phi_{rr})r^{n-1}\,dr-\int_1^\infty |\phi_t|^2r^{n-1}\,dr.
		\end{aligned}
	\end{align}
	We substitute \eqref{est:cross-01} into \eqref{est:cross-00} to obtain
	\begin{align}
		\begin{aligned}\label{est:cross-1}
		\frac{d}{dt}&\int_1^\infty \rho\psi \phi_r r^{n-1}\,dr \\
		&=\int_1^\infty \rho\psi^2\phi_{rr}r^{n-1}\,dr+\int_1^\infty|\phi_t|^2r^{n-1}\,dr-\int_1^{\infty}(P(\rho)-P(\tilde{\rho}))_r\phi_r r^{n-1}\,dr\\
		&\quad +\mu\int_1^\infty L(\psi)\phi_r r^{n-1}\,dr +\kappa\int_1^\infty \rho A(\phi)_r \phi_r r^{n-1}\,dr+\int_1^\infty h'(\tilde{\rho})\tilde{\rho}_r \phi\phi_r r^{n-1}\,dr.
		\end{aligned}
	\end{align}
	We further estimate \eqref{est:cross-1} to extract dissipation terms. First of all, we observe that
	\begin{align}
		\begin{aligned}\label{est:P1}
		\int_1^\infty (P(\rho)-P(\tilde{\rho}))_r\phi_r r^{n-1}\,dr&=\int_1^\infty (P'(\rho)\rho_r \phi_r -P'(\tilde{\rho})\tilde{\rho}_r\phi_r)r^{n-1}\,dr \\
		&= \int_1^\infty P'(\rho)\phi_r^2r^{n-1}\,dr + \int_1^\infty (P'(\rho)-P'(\tilde{\rho}))\tilde{\rho}_r\phi_r r^{n-1}\,dr.
		\end{aligned}
	\end{align}
	Next, since $L(\psi) = G_r$, we use integration by parts to get 
	\begin{equation}\label{est:G1}
		\mu\int_1^\infty L(\psi) \phi_r r^{n-1}\,dr=-\mu\int_1^\infty G\phi_{rr}r^{n-1}\,dr -\mu(n-1)\int_1^\infty G\phi_r r^{n-2}\,dr. 
	\end{equation}
	Finally, we again use integration by parts and $(\phi_r r^{n-1})_r = A(\phi)r^{n-1}$ to get
	\begin{align}
		\begin{aligned}\label{est:A1}
		\kappa\int_1^\infty\rho A(\phi)_r\phi_r r^{n-1}\,dr&=-\kappa\int_1^\infty A(\phi)(\rho\phi_r r^{n-1})_r\,dr\\
		&=-\kappa\int_1^\infty A(\phi)\left(\rho_r \phi_r r^{n-1}+\rho A(\phi)r^{n-1}\right)\,dr\\
		&=-\kappa\int_1^\infty \rho A(\phi)^2 r^{n-1}\,dr -\kappa\int_1^\infty A(\phi)\rho_r \phi_r r^{n-1}\,dr.
		\end{aligned}
	\end{align}
	Combining the above estimates \eqref{est:P1}, \eqref{est:G1}, and \eqref{est:A1}, the equation \eqref{est:cross-1} can be written as
	\begin{align*}
		&\frac{d}{dt}\int_1^\infty\rho\psi\phi_r r^{n-1}\,dr +\int_1^\infty P'(\rho)\phi_r^2r^{n-1}\,dr+\kappa\int_1^\infty \rho A(\phi)^2r^{n-1}\,dr\\
		&=\int_1^\infty \rho\psi^2\phi_{rr}r^{n-1}\,dr+\|\phi_t\|^2-\int_1^\infty(P'(\rho)-P'(\tilde{\rho}))\tilde{\rho}_r\phi_r r^{n-1}\,dr-\mu\int_1^\infty G\phi_{rr}r^{n-1}\,dr\\
		&\quad -\mu(n-1)\int_1^\infty G\phi_r r^{n-2}\,dr -\kappa\int_1^\infty A(\phi)\rho_r\phi_r r^{n-1}\,dr +\int_1^\infty h'(\tilde{\rho})\tilde{\rho}_r\phi\phi_r r^{n-1}\,dr\\
		&=:\sum_{i=1}^7 I_i.
	\end{align*}
	Below, we estimate each $I_i$ on the right-hand side separately. \\
	
	\noindent $\bullet$ (Estimate of $I_1$): We use Lemma \ref{lem:Hardy} to observe $\|\psi\|_{L^\infty}\le C\|\psi_r\|$. Then, we use the boundedness of $\rho$, the {\it a priori} smallness assumption \eqref{smallness_assumption}, and $\|\psi_r\|^2\le \|G\|^2\le \mathcal{D}$ to get
	\[|I_1|\le C\|\psi\|_{L^\infty}\|\psi\|\|\phi_{rr}\|\le C\e\|\psi_r\|\|\phi_{rr}\|\le \tau \|\phi_{rr}\|^2 +C\e^2\mathcal{D},\]
	where $\tau$ is a small constant that will be chosen later.\\
	
	\noindent $\bullet$ (Estimate of $I_2$): It follows from \eqref{eq:perturbation}$_1$ that $\phi_t = -\rho G -\rho_r\psi$. Therefore, we again use the boundedness of $\rho$ and $\rho_r=\tilde{\rho}_r+\phi_r$ to get
	\begin{align*}
		I_2&=\|\phi_t\|^2\le C\|G\|^2+C\int_1^\infty (\tilde{\rho}_r+\phi_r)^2\psi^2 r^{n-1}\,dr\\
		&\le C\|G\|^2+C\int_1^\infty \tilde{\rho}_r^2\psi^2 r^{n-1}\,dr+C\int_1^\infty \phi_r^2\psi^2 r^{n-1}\,dr.
	\end{align*}
	Using $|\tilde{\rho}_r|\le C|\rho_b|e^{-\sigma r}\le C\delta e^{-\sigma r}$ and Lemma \ref{lem:exp-Hardy}, we get
	\[\int_1^\infty \tilde{\rho}_r^2 \psi^2 r^{n-1}\,dr \le C\delta^2 \int_1^\infty |\psi|^2 e^{-2\sigma r}r^{n-1}\,dr\le C\delta^2 \|\psi_r\|^2\le C\delta^2\mathcal{D}.\]
	We also use Lemma \ref{lem:Hardy} and smallness assumption \eqref{smallness_assumption} to derive
	\[\int_1^\infty \phi_r^2\psi^2 r^{n-1}\,dr\le \|\psi\|_{L^\infty}^2 \|\phi_r\|^2\le C\|\psi_r\|^2\|\phi_r\|^2\le CN^2\|\psi_r\|^2\le C\e^2\mathcal{D}. \]
	Thus, we estimate $I_2$ as
	\begin{equation}\label{est:phi_t}
		I_2\le C\|G\|^2+C(\delta^2+\e^2)\mathcal{D}.
	\end{equation}
	\noindent $\bullet$ (Estimate of $I_3$): For $I_3$, we again use Lemma \ref{lem:exp-Hardy} to derive
	\begin{align*}
		|I_3|&\le C|\rho_b|\int_1^\infty |\phi||\phi_r|e^{-\sigma r}r^{n-1}\,dr\\
		&\le C\delta\left(\int_1^\infty|\phi|^2e^{-2\sigma r}r^{n-1}\,dr\right)^{1/2}\|\phi_r\|\le C\delta\|\phi_r\|^2\le C\delta\mathcal{D}.
	\end{align*}
	\noindent $\bullet$ (Estimate of $I_4$): We simply estimate $I_4$ as 
	\[|I_4|\le\mu \|G\|\|\phi_{rr}\|\le\tau\|\phi_{rr}\|^2+ C\|G\|^2.\]
	
	\noindent $\bullet$ (Estimate of $I_5$): 
	For $I_5$, since $r\ge1$ and $n\ge3$, we have $r^{n-2}\le r^{n-1}$ and therefore
	\[|I_5|\le C\int_1^\infty |G||\phi_r|r^{n-1}\,dr\le C\|G\|\|\phi_r\|\le \tau\|\phi_r\|^2 +C\|G\|^2. \]
	
	\noindent $\bullet$ (Estimate of $I_6$):
	To estimate $I_6$, we first note that
	\[\|\tilde{\rho}_r\phi_r\|^2\le C|\rho_b|^2\int_1^\infty \phi_r^2 e^{-2\sigma r}r^{n-1}\,dr\le C\delta^2\|\phi_r\|^2.\]
	Using this, we estimate $I_6$ as
	\begin{align*}
		|I_6|&\le C\int_1^\infty |A(\phi)||\rho_r||\phi_r| r^{n-1}\,dr \le C\int_1^\infty |A(\phi)||\tilde{\rho}_r||\phi_r|r^{n-1}\,dr+C\int_1^\infty |A(\phi)||\phi_r|^2r^{n-1}\,dr\\
		&\le C\|A(\phi)\|\|\tilde{\rho}_r\phi_r\|+C\|A(\phi)\|\|\phi_r\|\|\phi_r\|_{L^\infty}\le \tau \|A(\phi)\|^2+C\|\tilde{\rho}_r\phi_r\|^2+ C\|\phi_r\|_{L^\infty}^2\|\phi_r\|^2\\
		&\le \tau\|A(\phi)\|^2+C\delta^2\|\phi_r\|^2+CN^2\|\phi_{rr}\|^2\le \tau\|A(\phi)\|^2+C(\delta^2+\e^2)\mathcal{D},
	\end{align*}
	where we apply Lemma \ref{lem:Hardy} to $\phi_r$ to get $\|\phi_r\|_{L^\infty}\le C\|\phi_{rr}\|$.\\
	
	\noindent $\bullet$ (Estimate of $I_7$): Finally, for $I_7$, we use the same argument as in the estimate of $I_3$ to get
	\[|I_7|\le C\delta\|\phi_r\|^2 \le C\delta\mathcal{D}.\]
	Combining all the estimates on $I_i$ for $i=1,2,\ldots, 7$, we get
	\begin{align*}
		\frac{d}{dt}&\int_1^\infty\rho\psi\phi_r r^{n-1}\,dr +\int_1^\infty P'(\rho)\phi_r^2r^{n-1}\,dr+\kappa\int_1^\infty \rho A(\phi)^2r^{n-1}\,dr\\
			&\le \tau \|A(\phi)\|^2+ \tau\|\phi_r\|^2  + 2\tau\|\phi_{rr}\|^2 +C\|G\|^2 +C(\delta+\e^2)\mathcal{D}.
	\end{align*}
	On the other hand, since $A(\phi)=\phi_{rr}+\frac{n-1}{r}\phi_r$ and $\phi_r(1)=0$, we use Lemma \ref{lem:separate} to get
	\begin{align*}
		\int_1^\infty A(\phi)^2 r^{n-1}\,dr=\|\phi_{rr}\|^2+(n-1)\|\phi_r\|^2_{L^2_{r^{n-3}}}.
	\end{align*}
	Since $\rho$ is bounded below by $\rho_m>0$, the dissipation terms are bounded below as
	\begin{align*}
		&\int_1^\infty P'(\rho)\phi_r^2 r^{n-1}\,dr \ge P'(\rho_m)\|\phi_r\|^2,\\
		&\kappa\int_1^\infty \rho A(\phi)^2r^{n-1}\,dr\ge \kappa\rho_m\int_1^\infty A(\phi)^2r^{n-1}\,dr= \kappa\rho_m \left(\|\phi_{rr}\|^2+(n-1)\|\phi_r\|_{L^2_{r^{n-3}}}^2\right).
	\end{align*} 
	Therefore, choosing $\tau$ sufficiently small so that 
	\[\tau<\min\left\{\frac{\kappa\rho_m}{8},\frac{P'(\rho_m)}{2}\right\},\]
	we obtain
	\[\frac{d}{dt}\int_1^\infty \rho\psi\phi_rr^{n-1}\,dr+\frac{P'(\rho_m)}{2}\|\phi_r\|^2+\frac{\kappa\rho_m}{2}\|\phi_{rr}\|^2\le C\|G\|^2 +C(\delta+\e^2)\mathcal{D}.\]
	Thus, taking $c_2=\min\left\{\frac{P'(\rho_m)}{2},\frac{\kappa\rho_m}{2}\right\}$, we get the desired estimate.
\end{proof}

\subsection{High-order estimate}

Next, we conduct estimates on the higher derivatives of $\phi$ and $\psi$.

\begin{lemma}\label{lem:H1}
	Under the same assumptions as in Proposition \ref{prop:apriori}, there exists a positive constant $C_3$ such that
	\[\frac{d}{dt}\int_1^\infty \left(\frac{\rho}{2}|\psi_r|^2+\frac{\kappa}{2}|\phi_{rr}|^2\right)r^{n-1}\,dr +\frac{3\mu}{4}\|G_r\|^2\le C_3(\|\phi_{rr}\|^2+\|\phi_r\|^2+\|\psi_r\|^2)+C(\delta+\e)\mathcal{D}.\]
\end{lemma}
\begin{proof}
	We differentiate \eqref{eq:perturbation}$_2$ with respect to $r$, multiply by $\psi_r r^{n-1}$ and then integrate over $(1,\infty)$ to get
	\begin{align*}
		&\underbrace{\<(\rho\psi_t)_r,\psi_r\>+\<(\rho\psi\psi_r)_r,\psi_r\>+\<(P-\tilde{P})_{rr},\psi_r\>}_{=:L_1+L_2+L_3}\\
		&\quad=\underbrace{\mu\<G_{rr},\psi_r\>+\kappa\<(\rho A(\phi)_r)_r,\psi_r\>+\<(h'(\tilde{\rho})\tilde{\rho}_r\phi)_r,\psi_r\>}_{=:R_1+R_2+R_3}.
	\end{align*}
	First, we observe that $L_1$ and $L_2$ can be modified as
	\begin{align*}
		L_1 &= \<\rho_r\psi_t,\psi_r\>+\<\rho\psi_{tr},\psi_r\>=\<\rho_r\psi_t,\psi_r\>+\frac{1}{2}\frac{d}{dt}\|\sqrt{\rho}\psi_r\|^2-\frac{1}{2}\<\rho_t,\psi_r^2\>\\
		&=\<\rho_r\psi_t,\psi_r\>+\frac{1}{2}\frac{d}{dt}\|\sqrt{\rho}\psi_r\|^2-\<\rho\psi\psi_r,\psi_{rr}\>,
	\end{align*}
	and
	\[L_2=\<\rho_r\psi,\psi_r^2\>+\<\rho,\psi_r^3\>+\<\rho\psi\psi_r,\psi_{rr}\>.\]
	Therefore, we get
	\[L_1+L_2=\frac{1}{2}\frac{d}{dt}\|\sqrt{\rho}\psi_r\|^2+\<\rho_r\psi_t,\psi_r\>+\<\rho_r\psi,\psi_r^2\>+\<\rho,\psi_r^3\>.\]
	For $L_3$, we take integration by parts to get
	\[L_3 = -(P-\tilde{P})_r(1)\psi_r(1)-\<(P-\tilde{P})_r,\psi_{rr}\>-(n-1)\int_1^\infty (P-\tilde{P})_r\psi_r r^{n-2}\,dr.\]
	Next, we focus on the right-hand side terms. First, using integration by parts, we have
	\begin{align*}
		R_1 = -\mu G_r(1)\psi_r(1)-\mu\<G_r,\psi_{rr}\>-\mu(n-1)\int_1^\infty G_r\psi_r r^{n-2}\,dr.
	\end{align*}
	It follows from the definition of $G$ that  $\psi_{rr}=G_r-\frac{n-1}{r}\psi_r+\frac{n-1}{r^2}\psi$, and we substitute it to get
	\[-\mu\<G_r,\psi_{rr}\>=-\mu \|G_r\|^2+\mu(n-1)\int_1^\infty G_r\psi_rr^{n-2}\,dr -\mu(n-1)\int_1^\infty G_r\psi r^{n-3}\,dr.\]
	Hence, $R_1$ can be written as
	\[R_1 = -\mu G_r(1)\psi_r(1)-\mu\|G_r\|^2-\mu(n-1)\int_1^\infty G_r\psi r^{n-3}\,dr.\]
	Next, we consider $R_2$. Again, we use integration by parts to observe
	\begin{align*}
		R_2=-\kappa\rho(1) A(\phi)_r(1)\psi_r(1)-\kappa\<\rho A(\phi)_r,\psi_{rr}\>-\kappa(n-1)\int_1^\infty \rho A(\phi)_r \psi_r r^{n-2}\,dr.
	\end{align*}
	Again, using the definition of $G$, we get $\rho\psi_{rr} = \rho G_r-\frac{n-1}{r}\rho\psi_r +\frac{n-1}{r^2}\rho\psi$. On the other hand, by differentiating \eqref{eq:perturbation}$_1$ with respect to $r$, we get $\rho G_r=-\phi_{tr}-\rho_r G-\rho_{rr}\psi-\rho_r\psi_r$. These equations imply the following identity:
	\[\rho\psi_{rr} = -\phi_{tr}-\rho_rG-\rho_{rr}\psi-\rho_r\psi_r-\frac{n-1}{r}\rho\psi_r+\frac{n-1}{r^2}\rho\psi.\]
	Hence, we reformulate the second term of $R_2$ as
	\begin{align*}
		\<\rho A(\phi)_r,\psi_{rr}\>&=-\< A(\phi)_r,\phi_{tr}\>-\<A(\phi)_r,\rho_r G\>-\<A(\phi)_r,\rho_{rr}\psi\>-\<A(\phi)_r,\rho_r\psi_r\>\\
		&\quad -(n-1)\int_1^\infty \rho A(\phi)_r\psi_r r^{n-2}\,dr + (n-1)\int_1^\infty \rho A(\phi)_r\psi r^{n-3}\,dr,
	\end{align*}
	and therefore,
	\begin{align*}
		R_2&=-\kappa\rho(1)A(\phi)_r(1)\psi_r(1)+\kappa\<A(\phi)_r,\phi_{tr}\>+\kappa\<A(\phi)_r,\rho_rG\>+\kappa\<A(\phi)_r,\rho_{rr}\psi\>\\
		&\quad+\kappa\<A(\phi)_r,\rho_r\psi_r\> -\kappa(n-1)\int_1^\infty \rho A(\phi)_r\psi r^{n-3}\,dr.
	\end{align*}
	Using $A(\phi)=\phi_{rr}+\frac{n-1}{r}\phi_r$, the term $\kappa\<A(\phi)_r,\phi_{tr}\>$ in $R_2$ can be represented as
	\begin{align*}
	\kappa\<A(\phi)_r,\phi_{tr}\> &= -\kappa\<A(\phi),\phi_{trr}\>-\kappa(n-1)\int_1^\infty \left(\phi_{rr}+\frac{n-1}{r}\phi_r\right)\phi_{tr}r^{n-2}\,dr\\
	&=-\kappa\<\phi_{rr},\phi_{trr}\> -\kappa(n-1)\int_1^\infty\phi_{r}\phi_{trr}r^{n-2}\,dr\\
	&\quad -\kappa(n-1)\int_1^\infty \left(\phi_{rr}+\frac{n-1}{r}\phi_r\right)\phi_{tr}r^{n-2}\,dr\\
	&= -\frac{\kappa}{2}\frac{d}{dt}\|\phi_{rr}\|^2-\kappa(n-1)\int_1^\infty \phi_r\phi_{tr}r^{n-3}\,dr.
	\end{align*}
	From the estimates on $L_i$ and $R_i$, we get 
	\begin{align*}
		\frac{1}{2}\frac{d}{dt}&(\|\sqrt{\rho}\psi_r\|^2+\kappa\|\phi_{rr}\|^2) +\mu\|G_r\|^2\\
		&=\underbrace{-\mu G_r(1)\psi_r(1)-\kappa\rho(1)A(\phi)_r(1)\psi_r(1)+(P-\tilde{P})_r(1)\psi_r(1)}_{=:B}\\
		&\quad -\<\rho_r\psi_t,\psi_r\>-\<\rho_r\psi,\psi_r^2\>-\<\rho,\psi_r^3\>+\<(P-\tilde{P})_{r},\psi_{rr}\>+(n-1)\int_1^\infty (P-\tilde{P})_r\psi_rr^{n-2}\,dr\\
		&\quad -\mu(n-1)\int_1^\infty G_r\psi r^{n-3}\,dr-\kappa(n-1)\int_1^\infty \phi_r\phi_{tr}r^{n-3}\,dr+\kappa\<A(\phi)_r,\rho_rG\>+\kappa\<A(\phi)_r,\rho_{rr}\psi\>\\
		&\quad +\kappa\<A(\phi)_r,\rho_r\psi_r\> -\kappa(n-1)\int_1^\infty \rho A(\phi)_r\psi r^{n-3}\,dr+\<(h'(\tilde{\rho})\tilde{\rho}_r\phi)_r,\psi_r\>\\
		&=:B+\sum_{i=1}^{12}J_{i},
	\end{align*}
	where $B$ represents the terms for the boundary.\\
	
	\noindent $\bullet$ (Estimate of $B$): Evaluating \eqref{eq:perturbation}$_2$ at $r=1$ and using $\psi(1)=\psi_t(1)=0$, we get
	\[0=\mu G_r(1)+\kappa\rho(1)A(\phi)_r(1)+h'(\tilde{\rho}(1))\tilde{\rho}_r(1)\phi(1)-(P-\tilde{P})_r(1).\]
	Therefore, $B$ can be simplified as
	\[B=\psi_r(1)\left(-\mu G_r(1)-\kappa\rho(1)A(\phi)_r(1)+(P-\tilde{P})_r(1)\right)=\psi_r(1)h'(\tilde{\rho}(1))\tilde{\rho}_r(1)\phi(1).\]
	We now use Lemma \ref{lem:Hardy} to estimate $B$ as  
	\begin{align*}
		|B|&=|\psi_r(1)h'(\tilde{\rho}(1))\tilde{\rho}_r(1)\phi(1)|\le C|\rho_b||\phi(1)||\psi_r(1)|\le C\delta\|\phi_r\|\|\psi_{rr}\|.
	\end{align*}
	On the other hand, since $\psi(1)=0$, we use $\psi_{rr} = G_r-\frac{n-1}{r}\psi_r+\frac{n-1}{r^2}\psi$, $r\ge1$, and Lemma \ref{lem:Hardy} to get
	\begin{align}
		\begin{aligned}\label{psi_rr}
		\|\psi_{rr}\|^2&\le C\|G_r\|^2+C\int_1^\infty \psi_r^2 r^{n-3}\,dr + C\int_1^\infty \psi^2 r^{n-5}\,dr\\
		&\le C\|G_r\|^2+C\|\psi_r\|^2 + C\|\psi\|_{L^2_{r^{n-3}}}^2\le C\|G_r\|^2+C\|\psi_r\|^2.
		\end{aligned}
	\end{align}
	Therefore, we further estimate $B$ as
	\[B\le C\delta\|\phi_r\|(\|G_r\|+\|\psi_r\|)\le \frac{\mu}{8}\|G_r\|^2 +C\delta\mathcal{D},\]
	since $\|\psi_r\|^2 \le \|G\|^2\le \mathcal{D}$.\\
	
	\noindent $\bullet$ (Estimate of $J_1$ and $J_{10}$): Using \eqref{eq:perturbation}$_2$, $J_{10}$ can be expanded as
	\begin{align*}
		J_{10} &= \kappa\<\rho_r A(\phi)_r,\psi_r\>=\int_1^\infty \frac{\rho_r}{\rho}\left(\rho\psi_t -\mu G_r +\rho\psi\psi_r +(P-\tilde{P})_r-h'(\tilde{\rho})\tilde{\rho}_r\phi\right)\psi_r r^{n-1}\,dr\\
		&=-J_1 -\mu\int_1^\infty \frac{\rho_r}{\rho}G_r\psi_r r^{n-1}+\<\rho_r\psi,\psi_r^2\>+\int_1^\infty \frac{\rho_r}{\rho}(P-\tilde{P})_r\psi_r r^{n-1}\,dr\\
		&\quad -\int_1^\infty \frac{\rho_r}{\rho} h'(\tilde{\rho})\tilde{\rho}_r\phi\psi_r r^{n-1}\,dr\\
		&\le -J_1 +C\|\rho_r\|_{L^\infty}\|G_r\|\|\psi_r\|+\|\rho_r\|_{L^\infty}\|\psi\|_{L^\infty}\|\psi_r\|^2+C\|\rho_r\|_{L^\infty}\|\phi_r\|\|\psi_r\|\\
		&\quad + C\delta\|\rho_r\|_{L^\infty}\|\phi e^{-\sigma r}\|\|\psi_r\|\\
		&\le -J_1 + C(\delta+\e)\mathcal{D},
	\end{align*}
	where we used $\|\rho_r\|_{L^\infty}\le \|\tilde{\rho}_r\|_{L^\infty}+\|\phi_r\|_{L^\infty}\le C(\delta+\e)$ and $\|\phi e^{-\sigma r}\|\le C\|\phi_r\|$, thanks to Lemma \ref{lem:exp-Hardy}. Thus, we have
	\[J_1+J_{10}\le C(\delta+\e)\mathcal{D}.\]
	\noindent $\bullet$ (Estimate of $J_2$): We again use $\|\rho_r\|_{L^\infty}\le \|\phi_r\|_{L^\infty}+\|\tilde{\rho}_r\|_{L^\infty}\le C(\delta+\e)$ and $\|\psi\|_{L^\infty}<N<\e$ to estimate $J_2$ as
	\[J_2\le \|\rho_r\|_{L^\infty}\|\psi\|_{L^\infty}\|\psi_r\|^2\le C(\delta+\e)\e\mathcal{D}.\]
	\noindent $\bullet$ (Estimate of $J_3$): We use Lemma \ref{lem:Hardy} and \eqref{psi_rr} to get $\|\psi_r\|_{L^\infty}\le C\|\psi_{rr}\|\le C(\|G_r\|+\|\psi_r\|)$. This yields
	\begin{align*}
		J_3&\le C\|\psi_r\|_{L^\infty}\|\psi_r\|^2\le C(\|G_r\|+\|\psi_r\|)\|\psi_r\|^2\\
		&\le \frac{\mu}{32}\|G_r\|^2 +C\|\psi_r\|^{4}+C\|\psi_r\|^3\le \frac{\mu}{32}\|G_r\|^2+C\e\mathcal{D}.
	\end{align*}
	\noindent $\bullet$ (Estimate of $J_4$): Again, it follows from \eqref{psi_rr} that
	\begin{align*}
		J_4&\le C\|\phi_r\|\|\psi_{rr}\|\le C\|\phi_r\|\|G_r\|+C\|\phi_r\|\|\psi_r\|\le \frac{\mu}{32}\|G_r\|^2+C\|\phi_r\|^2+C\|\psi_r\|^2.
	\end{align*}
	\noindent $\bullet$ (Estimate of $J_5$): Since $r\ge 1$, we simply bound $J_5$ as
	\[J_5\le C(\|\phi_r\|^2+\|\psi_r\|^2).\]
	\noindent $\bullet$ (Estimate of $J_6$): As $\psi(1)=0$, we use Lemma \ref{lem:Hardy} to get
	\[J_6\le C\|G_r\|\|\psi\|_{L^2_{r^{n-3}}}\le C\|G_r\|\|\psi_r\|\le \frac{\mu}{32}\|G_r\|^2+C\|\psi_r\|^2.\]
	\noindent $\bullet$ (Estimate of $J_7$): Using $\phi_{tr} = -(\rho_r G+\rho G_r+\rho_{rr}\psi+\rho_r\psi_r)$, we split $J_7$ as
	\begin{align*}
		J_7 &=\kappa(n-1)\int_1^\infty \phi_r(\rho_rG+\rho G_r+\rho_{rr}\psi+\rho_r\psi_r)r^{n-3}\,dr=:\sum_{i=1}^4J_{7i}.
	\end{align*}
	We again use $\|\rho_r\|_{L^\infty}<C(\delta+\e)$ and $r\ge 1$ to get
	\[J_{71}\le C\|\rho_r\|_{L^\infty}\|\phi_r\|\|G\|\le C(\delta+\e)\|\phi_{r}\|\|G\|\le C(\delta+\e)\mathcal{D}.\]
	Again, we use $r\ge 1$ and Young's inequality to get
	\begin{align*}
		J_{72}\le C\|\phi_r\|\|G_r\|\le \frac{\mu}{32}\|G_r\|^2+C\|\phi_{r}\|^2.
	\end{align*}
	For $J_{73}$, we split $\rho_{rr}=\tilde{\rho}_{rr}+\phi_{rr}$ and use Lemma \ref{lem:Hardy} and Lemma \ref{lem:exp-Hardy} to estimate it as
	\begin{align*}
		J_{73} &\le \int_1^\infty|\phi_r||\tilde{\rho}_{rr}||\psi| r^{n-3}\,dr+\int_1^\infty |\phi_r||\phi_{rr}||\psi|r^{n-3}\,dr\\
		& \le C|\rho_b|\|\phi_r\|\|\psi e^{-\sigma r}\|+\|\phi_r\|_{L^\infty}\|\phi_{rr}\|\|\psi\|_{L^2_{r^{n-3}}}\le C\delta\|\phi_r\|\|\psi_r\|+C\e\|\phi_{rr}\|\|\psi_r\|\\
		&\le C(\delta+\e)\mathcal{D}.
	\end{align*}
	Finally, it is straightforward to get
	\[J_{74}\le C\|\rho_r\|_{L^\infty}\|\phi_r\|\|\psi_r\|\le C(\delta+\e)\|\phi_{r}\|\|\psi_r\|\le C(\delta+\e)\mathcal{D}.\]
	Combining the estimates on $J_{7i}$, we get
	\[J_7\le \frac{\mu}{32}\|G_r\|^2+C\|\phi_{r}\|^2+C(\delta+\e)\mathcal{D}.\]
	\noindent $\bullet$ (Estimate of $J_8$): To estimate $J_8$, we use \eqref{eq:perturbation}	to get
	\[J_8 = \int_1^\infty \rho_rG\left(\psi_t+\psi\psi_r -\frac{\mu}{\rho}G_r + \frac{1}{\rho}(P-\tilde{P})_r -\frac{1}{\rho}h'(\tilde{\rho})\tilde{\rho}_r\phi\right)r^{n-1}\,dr=\sum_{i=1}^5J_{8i}.\]
	We again use $\|\rho_r\|_{L^\infty}\le C(\delta+\e)$ to bound each term $J_{8i}$ as
	\begin{align*}
		J_{81}&\le C(\delta+\e)\|G\|\|\psi_t\|\le C(\delta+\e)\mathcal{D},\\
		J_{82}&\le C(\delta+\e)\|\psi\|_{L^\infty}\|G\|\|\psi_r\|\le C(\delta+\e)\e\mathcal{D},\\
		J_{83}&\le C(\delta+\e)\|G\|\|G_r\|\le C(\delta+\e)\mathcal{D},\\
		J_{84}&\le C(\delta+\e)\|G\|\|\phi_r\|\le C(\delta+\e)\mathcal{D},\\
		J_{85}&\le C(\delta+\e)\delta\|\phi e^{-\sigma r}\|\|G\|\le C(\delta+\e)\delta\|\phi_r\|\|G\|\le C(\delta+\e)\delta\mathcal{D}.
	\end{align*}
	Combining the above estimates, we get $J_8\le C(\delta+\e)\mathcal{D}$.\\
	
	\noindent $\bullet$ (Estimate of $J_9$): Similar to the estimate of $J_8$, we use \eqref{eq:perturbation}$_2$ to get
	\[J_9 = \int_1^\infty\rho_{rr}\psi\left(\psi_t+\psi\psi_r -\frac{\mu}{\rho}G_r + \frac{1}{\rho}(P-\tilde{P})_r -\frac{1}{\rho}h'(\tilde{\rho})\tilde{\rho}_r\phi\right)r^{n-1}\,dr=\sum_{i=1}^5 J_{9i}.\]
	Using $|\rho_{rr}|\le |\tilde{\rho}_{rr}|+|\phi_{rr}|\le C|\rho_b|e^{-\sigma r}+|\phi_{rr}|$ and Lemma \ref{lem:exp-Hardy}, we get 
	\begin{align*}
		J_{91}&\le C|\rho_b|\|\psi e^{-\sigma r}\|\|\psi_t\|+C\|\psi\|_{L^\infty}\|\phi_{rr}\|\|\psi_t\|\le C(\delta+\e)\mathcal{D},\\
		J_{92}&\le C|\rho_b|\|\psi\|_{L^\infty}\|\psi e^{-\sigma r}\|\|\psi_r\|+\|\psi\|_{L^\infty}^2\|\phi_{rr}\|\|\psi_r\|\le C(\delta+\e)\e\mathcal{D},\\
		J_{93}&\le C|\rho_b|\|G_r\|\|\psi e^{-\sigma r}\|+C\|\psi\|_{L^\infty}\|\phi_{rr}\|\|G_r\|\le C(\delta+\e)\mathcal{D},\\
		J_{94}&\le C|\rho_b|\|\phi_r\|\|\psi e^{-\sigma r}\|+C\|\psi\|_{L^\infty}\|\phi_{rr}\|\|\phi_r\|\le C(\delta+\e)\mathcal{D},\\
		J_{95}&\le C|\rho_b|^2\|\phi e^{-\sigma r}\|\|\psi e^{-\sigma r}\|+ C|\rho_b|\|\phi\|_{L^\infty}\|\phi_{rr}\|\|\psi e^{-\sigma r}\|\le C(\delta+\e)\delta\mathcal{D}.
	\end{align*}
	Hence, we get $J_{9}\le C(\delta+\e)\mathcal{D}$.\\
	
	\noindent $\bullet$ (Estimate of $J_{11})$:	Taking integration by parts, we split $J_{11}$ as
	\begin{align*}
		J_{11}& = \kappa(n-1)\int_1^\infty A(\phi)\pa_r(\rho \psi r^{n-3})\,dr \\
		&=\kappa(n-1)\int_1^\infty A(\phi)\rho_r \psi r^{n-3}\,dr +\kappa(n-1)\int_1^\infty A(\phi)\rho \psi_r r^{n-3}\,dr\\
		&\quad +\kappa(n-1)(n-3)\int_1^\infty A(\phi)\rho\psi r^{n-4}\,dr\\
		&=:J_{11,1}+J_{11,2}+J_{11,3}.
	\end{align*}
	We use $r\ge1$ and Lemma \ref{lem:Hardy} to estimate $J_{11,1}$ as
	\[J_{11,1}\le C\|\rho_r\|_{L^\infty}\|A(\phi)\|\|\psi\|_{L^2_{r^{n-3}}}\le C(\delta+\e)\|A(\phi)\|\|\psi_r\|\le C(\delta+\e)\mathcal{D},\]
	where we used $\phi_r(1)=0$ and Lemma \ref{lem:Hardy} to get
	\begin{equation}\label{est:Aphi}
		\|A(\phi)\|^2 = \|\phi_{rr}\|^2+(n-1)\|\phi_r\|_{L^2_{r^{n-3}}}^2\le C\|\phi_{rr}\|^2.
	\end{equation}
	Similarly, we use $r\ge 1$ to obtain 
	\[J_{11,2}\le C\|A(\phi)\|\|\psi_r\|\le C(\|\phi_{rr}\|^2+\|\psi_r\|^2).\]
	Finally, we use $r\ge 1$ and Lemma \ref{lem:Hardy} to estimate $J_{11,3}$ as 
	\[J_{11,3}\le C\|A(\phi)\|\|\psi\|_{L^2_{r^{n-3}}}\le C\|A(\phi)\|\|\psi_r\|\le C(\|\phi_{rr}\|^2+\|\psi_r\|^2).\]
	Thus, we derive $J_{11}\le C(\delta+\e)\mathcal{D}+C(\|\phi_{rr}\|^2+\|\psi_r\|^2)$. 
	
	\noindent $\bullet$ (Estimate of $J_{12}$): Finally, using Lemma \ref{lem:exp-Hardy}, it is easy to observe that
	\begin{align*}
		J_{12}&\le C|\rho_b|^2\int_1^\infty e^{-2\sigma r}|\phi||\psi_r|r^{n-1}\,dr + C|\rho_b|\int_1^\infty e^{-\sigma r}|\phi||\psi_r|r^{n-1}\,dr \\
		&\quad +C|\rho_b|\int_1^\infty e^{-\sigma r}|\phi_r||\psi_r|r^{n-1}\,dr\\
		&\le C\delta^2\|\phi e^{-\sigma r}\|\|\psi_r\|+C\delta \|\phi e^{-\sigma r}\|\|\psi_r\|+C\delta \|\phi_r\|\|\psi_r\|\le C\delta\mathcal{D}.
	\end{align*}
	
	Now, we gather all the estimates on $B$ and $J_{i}$ for $i=1,2,\ldots, 12$ to get
	\begin{align*}
	\frac{1}{2}\frac{d}{dt}(\|\sqrt{\rho}\psi_r\|^2+\kappa\|\phi_{rr}\|^2)+\frac{3\mu}{4}\|G_r\|^2\le C(\|\phi_{r}\|^2+\|\phi_{rr}\|^2+\|\psi_r\|^2)+C(\delta+\e)\mathcal{D},
	\end{align*}
	which is the desired estimate.
\end{proof}

To close the {\it a priori} estimate, we need to extract the dissipation term $\|\psi_t\|^2$, which is available from the following Lemma.

\begin{lemma}\label{lem:Phi}
	Under the same assumptions as in Proposition \ref{prop:apriori}, there exists a positive constant $C_4$ such that 
	\[\frac{d}{dt}\left(\frac{\mu}{2}\|G\|^2+\kappa\int_1^\infty \phi_r\phi_{tr}r^{n-1}\,dr\right)+\frac{5\rho_m}{8}\|\psi_t\|^2\le C_4( \|G_r\|^2+\|\phi_r\|^2)+C(\delta^2+\e)\mathcal{D}.\]
\end{lemma}

\begin{proof}
	We multiply \eqref{eq:perturbation}$_2$ by $\psi_t r^{n-1}$ and integrate over $(1,\infty)$ to get
	\begin{equation}\label{est:high-0}
		\|\sqrt{\rho}\psi_t\|^2+\<\rho\psi\psi_r,\psi_t\>+\<(P-\tilde{P})_r,\psi_t\>=\mu\<L(\psi) ,\psi_t\>+\kappa\<\rho A(\phi)_r,\psi_t\>+\<h'(\tilde{\rho})\tilde{\rho}_r\phi,\psi_t\>.
	\end{equation}
	The first term of the right-hand side of \eqref{est:high-0} becomes
	\begin{align*}
		\mu \<L(\psi),\psi_t\> = \mu\<G_r,\psi_t\>=-\mu\int_1^\infty G\left(\psi_{tr}+\frac{n-1}{r}\psi_t\right)r^{n-1}\,dr=-\mu\int_1^\infty GG_t r^{n-1}\,dr=-\frac{\mu}{2}\frac{d}{dt}\|G\|^2.
	\end{align*}
	Moreover, \eqref{eq:perturbation}$_1$ implies 
	\[\pa_r(\rho\psi_t r^{n-1})=-\phi_{tt}r^{n-1}-\pa_r(\phi_t\psi r^{n-1}).\]
	Using this, the second term of the right-hand side of \eqref{est:high-0} can be reformulated as
	\begin{align}
		\begin{aligned}\label{est:high-1}
		\kappa\<\rho A(\phi)_r,\psi_t\>&=-\kappa\int_1^\infty A(\phi)\pa_r(\rho\psi_t r^{n-1})\,dr\\
		&=\kappa\int_1^\infty A(\phi)\phi_{tt}r^{n-1}\,dr +\kappa\int_1^\infty A(\phi)\pa_r (\phi_t \psi r^{n-1})\,dr.
		\end{aligned}
	\end{align}
	By the definition of $A(\phi)$, the first term of the right-hand side of \eqref{est:high-1} can be written as
	\begin{align*}
		\kappa\int_1^\infty A(\phi)\phi_{tt}r^{n-1}\,dr&=\kappa\int_1^\infty \phi_{tt}\pa_r(\phi_r r^{n-1})\,dr=-\kappa\int_1^\infty \phi_r\phi_{ttr}r^{n-1}\,dr.\\
		&=-\kappa\frac{d}{dt}\<\phi_r,\phi_{tr}\>+\kappa\|\phi_{tr}\|^2,
	\end{align*}
	and the second term of the right-hand side of \eqref{est:high-1} is 
	\begin{align*}
		\kappa\int_1^\infty A(\phi)\pa_r(\phi_t\psi r^{n-1})\,dr=-\kappa\<A(\phi)_r,\phi_t\psi\>.
	\end{align*}
	Therefore, the equation \eqref{est:high-0} is rearranged as
	\begin{align*}
		&\|\sqrt{\rho}\psi_t\|^2+\frac{\mu}{2}\frac{d}{dt}\|G\|^2+\kappa\frac{d}{dt}\<\phi_r,\phi_{tr}\>\\
		&=-\<\rho\psi\psi_r,\psi_t\> -\<(P-\tilde{P})_r,\psi_t\>+\kappa\|\phi_{tr}\|^2-\kappa\<A(\phi)_r,\phi_t\psi\>+\<h'(\tilde{\rho})\tilde{\rho}_r\phi,\psi_t\>=:\sum_{i=1}^5 K_i.
	\end{align*}
	Below, we estimate each $K_i$ separately.\\
	
	\noindent $\bullet$ (Estimate of $K_1$): We use $\|\psi\|_{L^\infty}\le C\e$ and Young's inequality to get
	\[K_1\le C \|\psi\|_{L^\infty} \|\psi_r\|\|\psi_t\|\le C\e\|\psi_r\|\|\psi_t\|\le \frac{\rho_m}{8}\|\psi_t\|^2+C\e^2\mathcal{D}.\]
	
	\noindent $\bullet$ (Estimate of $K_2$): Again, we use Young's inequality to get
	\begin{align*}
		K_2 \le C\|\phi_r\|\|\psi_t\|\le \frac{\rho_m}{8}\|\psi_t\|^2 +C\|\phi_r\|^2.
	\end{align*}
	\noindent $\bullet$ (Estimate of $K_3$): We use \eqref{eq:perturbation}$_1$ and $\rho=\tilde{\rho}+\phi$ to estimate $K_3$ as
	\begin{align*}
		K_3&=\kappa\|\rho_rG+\rho G_r+\rho_{rr}\psi+\rho_r\psi_r\|^2\le C\kappa(\|\rho_rG\|^2+\|\rho G_r\|^2+\|\rho_{rr}\psi\|^2+\|\rho_r\psi_r\|^2)\\
		&\le C(\|\tilde{\rho}_rG\|^2+\|\phi_r G\|^2)+C\|G_r\|^2+C(\|\tilde{\rho}_{rr}\psi\|^2+\|\phi_{rr}\psi\|^2)+C(\|\tilde{\rho}_r\psi_r\|^2+\|\phi_r\psi_r\|^2).
	\end{align*}
	Now, using $|\tilde{\rho}_r|,|\tilde{\rho}_{rr}|\le C|\rho_b|<C\delta$, $\|\psi\|_{L^\infty}^2,\|\phi_r\|^2_{L^\infty}\le CN^2<C\e^2$, and Lemma \ref{lem:exp-Hardy}, we further estimate $K_3$ as
	\begin{align*}
		K_3&\le C\|G_r\|^2 + C\delta^2\|G e^{-\sigma r}\|^2 + C\e^2\|G\|^2+C\delta^2\|\psi e^{-\sigma r}\|^2+C\e^2\|\phi_{rr}\|^2+C(\delta^2+\e^2)\|\psi_r\|^2\\
		&\le C\|G_r\|^2 +C\delta^2\|G_r\|^2+C\e^2 \|G\|^2+C\delta^2\|\psi_r\|^2+C\e^2\|\phi_{rr}\|^2+C(\delta^2+\e^2)\|\psi_r\|^2\\
		&\le C\|G_r\|^2+C(\delta^2+\e^2)\mathcal{D}.
	\end{align*}
	\noindent $\bullet$ (Estimate of $K_4$): To estimate $K_4$, we first note that from \eqref{eq:perturbation}$_2$,
	\[\kappa \rho A(\phi)_r=\rho\psi_t +\rho\psi\psi_r+(P-\tilde{P})_r-\mu G_r -h'(\tilde{\rho})\tilde{\rho}_r\phi.\]
	Therefore, using the lower and upper bounds $\rho_m\le \rho\le \rho_M$, we have
	\begin{align*}
		\kappa^2\rho_m^2\|A(\phi)_r\|^2&\le C\left(\rho_M^2\|\psi_t\|^2+\rho_M^2\|\psi\psi_r\|^2+\|(P-\tilde{P})_r\|^2+\mu^2 \|G_r\|^2+\|h'(\tilde{\rho})\tilde{\rho}_r\phi\|^2\right)\\
		&\le C(\|\psi_t\|^2+\e^2\|\psi_r\|^2+\|\phi_r\|^2+\|G_r\|^2+\|\phi\tilde{\rho}_r\|^2)\\
		&\le C\mathcal{D},
	\end{align*}
	where we used $\|\psi\|_{L^\infty}<\e$ and $\|\phi\tilde{\rho}_r\|\le C|\rho_b|\|\phi_r\|$. This implies $\|A(\phi)_r\|\le C\sqrt{\mathcal{D}}$. Hence, we estimate $K_4$ as
	\begin{align*}
		K_4&= \kappa\<A(\phi)_r,\phi_t\psi\>\le \kappa\|A(\phi)_r\|\|\phi_t\psi\|\le C\|\psi\|_{L^\infty}\|\phi_t\|\sqrt{\mathcal{D}}\le C\e\mathcal{D}.
	\end{align*}
	Here, we use \eqref{est:phi_t} in the last inequality.\\
	
	\noindent $\bullet$ (Estimate of $K_5$): Finally, we use Lemma \ref{lem:exp-Hardy} and Young's inequality to estimate $K_5$ as
	\begin{align*}
		K_5&\le C|\rho_b|\|\phi e^{-\sigma r}\|\|\psi_t\|\le C\delta\|\phi_r\|\|\psi_t\|\le \frac{\rho_m}{8}\|\psi_t\|^2 +C\delta^2\mathcal{D}. 
	\end{align*}
	
	Gathering all the estimates on $K_i$, there exists a positive constant $C_4$ such that
	\begin{align*}
		\frac{d}{dt}\left(\frac{\mu}{2}\|G\|^2+\kappa\<\phi_r,\phi_{tr}\>\right)+\frac{5\rho_m}{8}\|\psi_t\|^2\le C_4(\|G_r\|^2+\|\phi_r\|^2) + C(\delta^2+\e)\mathcal{D}.
	\end{align*}
\end{proof}

\subsection{Proof of Proposition \ref{prop:apriori}}

Now, we are ready to prove Proposition \ref{prop:apriori}. We define energy functionals $\mathcal{E}_i$ for $i=1,2,3,4$ as 
\begin{align*}
	&\mathcal{E}_1:=\int_1^\infty \left(\frac{1}{2}\rho \psi^2+Q(\rho|\tilde{\rho})+\frac{\kappa}{2}\phi_r^2\right)r^{n-1}\,dr,\\
	&\mathcal{E}_2:=\<\rho\psi,\phi_r\>,\\
	&\mathcal{E}_3:=\int_1^\infty \left(\frac{1}{2}\rho \psi_r^2+\frac{\kappa}{2}\phi_{rr}^2\right)r^{n-1}\,dr,\\
	&\mathcal{E}_4:=\frac{\mu}{2}\|G\|^2+\kappa\<\phi_r,\phi_{tr}\>.
\end{align*}
Then, Lemma \ref{lem:L2}, Lemma \ref{lem:cross}, Lemma \ref{lem:H1}, and Lemma \ref{lem:Phi} imply the following estimates respectively:

\begin{align}
	\begin{aligned}\label{est:all}
	&\frac{d\mathcal{E}_1}{dt}+\mu \|G\|^2=0,\\
	&\frac{d\mathcal{E}_2}{dt} + c_2(\|\phi_r\|^2+\|\phi_{rr}\|^2)\le C_2\|G\|^2+C(\delta+\e^2)\mathcal{D},\\
	&\frac{d\mathcal{E}_3}{dt} + \frac{3\mu}{4}\|G_r\|^2\le C_3(\|\phi_{r}\|^2+\|\phi_{rr}\|^2+\|\psi_r\|^2)+C(\delta+\e)\mathcal{D},\\
	&\frac{d\mathcal{E}_4}{dt} + \frac{5\rho_m}{8}\|\psi_t\|^2\le C_4\left(\|G_r\|^2+\|\phi_r\|^2\right)+C(\delta^2+\e)\mathcal{D}.
	\end{aligned}
\end{align}
We now construct the combined energy functional 
\begin{equation}\label{def:L}
	\mathcal{L}:=\mathcal{E}_1+w_2\mathcal{E}_2+w_3\mathcal{E}_3+w_4\mathcal{E}_4,
\end{equation}
and choose the weights $w_i$ so that $\mathcal{L}$ is equivalent to $N^2$. From \eqref{est:all}, $\mathcal{L}$ satisfies

\begin{align*}
	&\frac{d\mathcal{L}}{dt}+ \mu \|G\|^2 +w_2 c_2(\|\phi_r\|^2+\|\phi_{rr}\|^2) + w_3 \frac{3\mu}{4}\|G_r\|^2+w_4\frac{5\rho_m}{8}\|\psi_t\|^2\\
	&\le w_2 C_2\|G\|^2 +w_3 C_3(\|\phi_{r}\|^2+\|\phi_{rr}\|^2+\|\psi_r\|^2)+ w_4 C_4(\|G_r\|^2+\|\phi_r\|^2)+C(\delta+\e)\mathcal{D}.
\end{align*}
First of all, we choose $w_2$ sufficiently small so that $w_2 C_2<\frac{\mu}{4}$. Then, we choose $w_3$ further small so that
\[w_3C_3<\min\left\{\frac{\mu}{4},\frac{w_2c_2}{4}\right\}.\]
Finally, we choose $w_4$ so that
\[w_4C_4<\min\left\{\frac{3w_3\mu}{8},\frac{w_2c_2}{4}\right\}.\]
These choices of $w_i$ imply that
\[\frac{d\mathcal{L}}{dt}+\frac{\mu}{2}\|G\|^2+\frac{w_2c_2}{2}(\|\phi_r\|^2+\|\phi_{rr}\|^2)+w_3\frac{3\mu}{8}\|G_r\|^2+w_4\frac{5\rho_m}{8}\|\psi_t\|^2\le C(\delta+\e)\mathcal{D}.\]
Then, it follows from the smallness of $\delta$ and $\e$, and the definition of $\mathcal{D}$ that there exists a positive constant $c_*$ such that for all $0\le t\le T$,
\[\frac{d\mathcal{L}}{dt}+c_*\mathcal{D}\le 0,\quad\mbox{therefore}\quad\mathcal{L}(t)+c_*\int_0^t\mathcal{D}(s)\,ds\le \mathcal{L}(0).\]
To complete the proof of Proposition \ref{prop:apriori}, it suffices to show that $\mathcal{L}$ is equivalent to $N^2$. We first show that $|\mathcal{E}_4|\le CN^2$. Since $r\ge 1$, it is directly observed that
\begin{equation}\label{est:G}
	\|G\|^2 = \|\psi_r\|^2+(n-1)\|\psi\|_{L^2_{r^{n-3}}}^2\le C(\|\psi\|^2+\|\psi_r\|^2)\le CN^2.
\end{equation}
To bound $\<\phi_r,\phi_{tr}\>$ in $\mathcal{E}_4$, we use \eqref{eq:perturbation}$_1$ to separate it as
\begin{align*}
	\<\phi_r,\phi_{tr}\>=-\<\phi_r,\rho_rG\>-\<\phi_r,\rho G_r\>-\<\phi_r,\rho_{rr}\psi\>-\<\phi_r,\rho_r\psi_r\>.
\end{align*} 
Using integration by parts, we get
\[-\<\phi_r,\rho G_r\>=\int_1^\infty G\pa_r(\rho\phi_r r^{n-1})\,dr=\<G,\rho A(\phi)+\rho_r\phi_r\>.\]
Hence, 
\[\<\phi_r,\phi_{tr}\>=\<\rho G,A(\phi)\>-\<\phi_r,\rho_{rr}\psi\>-\<\phi_r,\rho_r\psi_r\>.\]
To show that each term on the right-hand side is bounded by $CN^2$, we first recall from \eqref{est:Aphi} and \eqref{est:G} that
\[\|A(\phi)\|^2\le C\|\phi_{rr}\|^2\le CN^2,\quad \mbox{and}\quad \|G\|^2\le CN^2.\]
Thus, the first term is bounded as
\[\<\rho G,A(\phi)\> \le C\|G\|\|A(\phi)\|\le CN^2.\]
We estimate the second term by using Lemma \ref{lem:Hardy} as
\[\<\phi_r,\rho_{rr}\psi\>\le\|\phi_r\|_{L^\infty}(\|\tilde{\rho}_{rr}\|+\|\phi_{rr}\|)\|\psi\|\le C(\delta+\e)\|\phi_{rr}\|\|\psi\|\le C(\delta+\e)N^2.\]
Finally, the third term is bounded as
\[\<\phi_r,\rho_r\psi_r\>\le \|\phi_r\|(\|\phi_r\|_{L^\infty}+\|\tilde{\rho}_r\|_{L^\infty})\|\psi_r\|\le C(\delta+\e)N^2.\]
Collecting the above estimates and using \eqref{est:G} once again, we get
\[|\mathcal{E}_4|\le C\|G\|^2+C|\<\phi_r,\phi_{tr}\>|<C_*N^2.\]
Since the other terms in $\mathcal{E}_1$, $\mathcal{E}_2$ and $\mathcal{E}_3$ are easily bounded by $CN^2$, we conclude that $\mathcal{L}\le CN^2$. To show the reverse inequality, we use the lower and upper bounds of $\rho$ to obtain
\begin{align*}
	\mathcal{L}&\ge \frac{\rho_m}{2}\|\psi\|^2 + c_Q\|\phi\|^2 + \frac{\kappa}{2}\|\phi_r\|^2-\frac{w_2\rho_M}{2}\|\psi\|^2-\frac{w_2\rho_M}{2}\|\phi_r\|^2 + \frac{w_3\rho_m}{2}\|\psi_r\|^2+\frac{w_3\kappa}{2}\|\phi_{rr}\|^2-w_4 C_* N^2\\
	&=\left(\frac{\rho_m}{2}-\frac{w_2\rho_M}{2}-w_4C_*\right)\|\psi\|^2+(c_Q-w_4C_*)\|\phi\|^2\\
	&\quad +\left(\frac{\kappa}{2}-\frac{w_2\rho_M}{2}-w_4C_*\right)\|\phi_r\|^2+\left(\frac{w_3\rho_m}{2}-w_4C_*\right)\|\psi_r\|^2+\left(\frac{w_3\kappa}{2}-w_4C_*\right)\|\phi_{rr}\|^2,
\end{align*}
where $c_Q$ is the constant satisfying $Q(\rho|\tilde{\rho})\ge c_Q|\rho-\tilde{\rho}|^2$. Hence, once we choose the weights $w_2$ and $w_4$ further small so that
\[w_2<\min\left\{\frac{\rho_m}{4\rho_M},\frac{\kappa}{4\rho_M}\right\},\quad w_4<\min\left\{\frac{\rho_m}{8C_*},\frac{c_Q}{2C_*},\frac{\kappa}{8C_*},\frac{w_3\rho_m}{4C_*}, \frac{w_3\kappa}{4C_*}\right\},\]
we conclude that there exists another positive constant $c_{**}$ such that
\[\mathcal{L}\ge c_{**}\left(\|\phi\|^2+\|\psi\|^2+\|\phi_{r}\|^2+\|\phi_{rr}\|^2+\|\psi_r\|^2\right)=c_{**}N^2.\]
Therefore, for the above choice of the weights $w_2$, $w_3$, and $w_4$, we show that $\mathcal{L}$ is equivalent to $N^2$, and therefore, we obtain the desired {\it a priori} bound \eqref{est:apriori}. This completes the proof of Proposition \ref{prop:apriori}.

\section{Global well-posedness and time-asymptotic behavior}\label{sec:global}
\setcounter{equation}{0}

In this section, we prove the global well-posedness of the radially symmetric NSK equations \eqref{eq:NSK_spherical} when the initial perturbation and boundary data are sufficiently small. Furthermore, we prove that the global solution converges to the stationary solution as $t\to\infty$.

\subsection{Global well-posedness of \eqref{eq:NSK_spherical}}

The global well-posedness follows from the local well-posedness 
combined with the a priori estimate obtained in Proposition~\ref{prop:apriori}. 
The local-in-time well-posedness of the problem 
\eqref{eq:NSK_spherical}--\eqref{eq:boundary} can be established by the standard iteration scheme: one solves the linearized momentum equation with the capillary term, updates the density through the continuity equation, and shows the convergence of the approximate solutions via the energy method. Since the radial reduction turns \eqref{eq:NSK_spherical} into a one-dimensional system on the half-line $(1,\infty)$ with the weight $r^{n-1}$, the argument is parallel to the local existence theory for the initial-boundary value problems of the one-dimensional NSK equations (see, e.g., \cite{HKOpre, LTY22, LXC23, LZ21}) and for the exterior domain problem of the Navier--Stokes equations \cite{J96}; the compatibility condition stated in Section~1 is imposed precisely for this purpose. For the multi-dimensional local theory without the symmetry assumption, we refer to \cite{K08}. As the proof involves no new difficulty, we only 
state the result and omit the proof.

\begin{proposition}\label{prop:local}
	Let $n \ge 3$ and let $\tilde\rho$ be the stationary solution to the 
	impermeable wall problem satisfying \eqref{eq:stationary}. Suppose that 
	the initial data $(\rho_0, u_0)$ satisfies
	\[ (\rho_0 - \tilde\rho,\, u_0) \in H^3(1,\infty) \times H^2(1,\infty), \quad 
	0 < \underline{\rho} \le \rho_0(r) \le \overline{\rho}, \quad r \ge 1,\]
	for some positive constants $\underline{\rho}$ and $\overline{\rho}$, 
	together with the boundary conditions
	\[ u_0(1) = 0, \quad \rho_{0r}(1) = \rho_b,\]
	and the compatibility condition \eqref{eq:compatibility}. Then, there exists a positive constant $T_0$, depending only on 
	$\|(\rho_0 - \tilde\rho, u_0)\|_{H^3 \times H^2}$ and $\underline{\rho}$, 
	such that the initial-boundary value problem 
	\eqref{eq:NSK_spherical}--\eqref{eq:boundary} admits a unique solution $(\rho, u)$ satisfying
	\[(\rho - \tilde\rho,\, u) \in 
	C([0,T_0]; H^3(1,\infty)) \times C([0,T_0]; H^2(1,\infty))\]
	and 
	\[ \frac{\underline{\rho}}{2} < \rho(t,r) < 2\overline{\rho}, 
	\quad 0 \le t \le T_0, \quad r \ge 1.\]
\end{proposition}

In the following, we show that the local solution constructed in Proposition \ref{prop:local} can be extended to the global solution by the {\it a priori estimate} in Proposition \ref{prop:apriori} using the continuation argument. Suppose that the initial perturbation is small so that
\[N(0)<\e_0:=\frac{\e}{2\sqrt{C_0}},\]
where the constant $C_0$ is the same constant as in Proposition \ref{prop:apriori}. We define the function space
\[X(T):=C([0,T];H^3)\times C([0,T];H^2)\] 
and the maximal existence time $T_*$ as
\[T_*:=\sup\left\{T>0~\Big|~(\rho-\tilde{\rho},u)\in X(T),\quad N(t)<\e\quad \mbox{and}\quad \rho(t,r)> \frac{3}{4}\rho_+\quad \mbox{for all $t\in [0,T]$ and $r\ge 1$}\right\}.\]
Suppose that $T_*<+\infty$. Then, on the time interval $[0,T_*]$, we have $N(t)<\e$, and this implies
\[\|\rho(t)-\tilde{\rho}\|_{L^\infty}<CN(t)<C\e,\]
which guarantees $\rho(t,r)> \frac{3}{4}\rho_+$ for sufficiently small $\delta$ and $\e$. Furthermore, the {\it a priori} bound \eqref{est:apriori} implies
\[N(T_*)\le \sqrt{C_0}N(0)<\sqrt{C_0}\times\frac{\e}{2\sqrt{C_0}}<\frac{\e}{2}<\e.\]
Thus, if $T_*$ is finite, the only possible scenario is
\[\lim_{t\to T_*} \|(\rho-\tilde{\rho},u)\|_{H^3\times H^2}=+\infty.\]
Below, we will show that $\|(\phi,\psi)\|_{H^3\times H^2}$ is bounded on $[0,T_*)$ to yield a contradiction. Since $\|(\phi,\psi)\|_{H^2\times H^1}$ is already bounded by the {\it a priori} estimate, it suffices to show that $\|(\phi_{rrr},\psi_{rr})\|$ is bounded on $[0,T_*)$. However, it follows from \eqref{eq:perturbation}$_1$ that
\[\rho G_r = -\phi_{rt}-\rho_r G-\rho_{rr}\psi - \rho_r\psi_r, \]
which implies
\[\|G_r\|\le C\left(\|\phi_{rt}\|+\|\rho_r\|_{L^\infty}\|G\|+\|\rho_{rr}\|\|\psi\|_{L^\infty}+\|\rho_r\|_{L^\infty} \|\psi_r\|\right)\le C(\|\phi_{rt}\|+N).\]
Furthermore, we use \eqref{eq:perturbation}$_2$ to get
\[\kappa\rho A(\phi)_r = \rho\psi_t +\rho\psi\psi_r +(P-\tilde{P})_r-\mu G_r -h'(\tilde{\rho})\tilde{\rho}_r\phi.\]
This yields
\begin{equation}\label{est:Aphi_r}
	\|A(\phi)_r\|\le C(\|\psi_t\|+\|G_r\|+N)\le C(\|\psi_t\|+\|\phi_{rt}\|+N).
\end{equation}
Since $\psi_{rr} = G_r - \frac{n-1}{r}\psi_r + \frac{n-1}{r^2}\psi$ and $\phi_{rrr} = A(\phi)_r -\frac{n-1}{r}\phi_{rr}+\frac{n-1}{r^2}\phi_r$, we only need to show that $\|\psi_t\|$ and $\|\phi_{rt}\|$ are bounded on $[0,T_*)$, which is shown in the following lemma.

\begin{lemma}\label{lem:highest}
	Define the functional
	\[\mathcal{F}:=\int_1^\infty \left(\frac{\rho}{2}\psi_t^2+\frac{\kappa}{2}\phi_{rt}^2\right)r^{n-1}\,dr.\]
	Then, under the same assumptions as in Proposition \ref{prop:apriori}, 
	\[\mathcal{F}(t)+1\le(\mathcal{F}(0)+1)\exp\left(\int_0^{T_*}C(1+\|G_r(s)\|^2)\,ds\right)<+\infty\]
	for all $t\in [0,T_*)$.
\end{lemma}

\begin{proof}
	We first note that $\mathcal{F}(0)$ is finite since $(\phi_0,\psi_0)\in H^3\times H^2$, which implies $\psi_t(0), \phi_{rt}(0)\in L^2$. Now, it is straightforward to observe that
	\begin{align*}
		\frac{d\mathcal{F}}{dt} = \frac{1}{2}\<\rho_t\psi_t,\psi_t\>+\<\rho\psi_{tt},\psi_t\>+\kappa\<\phi_{rt},\phi_{rtt}\>.
	\end{align*}
	Taking time derivative on \eqref{eq:perturbation}$_2$ and taking inner product with $\psi_t$, we obtain
	\begin{align*}
		\<\rho\psi_{tt},\psi_t\> &= -\<\rho_t\psi_t ,\psi_t\>-\<\pa_t(\rho\psi\psi_r),\psi_t\>-\<\pa_t(P-\tilde{P})_r,\psi_t\>+\mu\<G_{rt},\psi_t\>\\
		&\quad+\kappa\<\pa_t(\rho A(\phi)_r),\psi_t\>+\<\pa_t(h'(\tilde{\rho})\tilde{\rho}_r\phi),\psi_t\>.
	\end{align*}
	However, note that
	\[\mu\<G_{rt},\psi_t\>=-\mu\int_1^\infty G_t \pa_r(\psi_t r^{n-1})\,dr=-\mu\|G_t\|^2.\]
	Therefore, we reformulate $\frac{d\mathcal{F}}{dt}$ as
	\begin{align*}
		\frac{d\mathcal{F}}{dt}+\mu\|G_t\|^2&=-\frac{1}{2}\<\rho_t\psi_t,\psi_t\>+\kappa\<\pa_t(\rho A(\phi)_r),\psi_t\>+\kappa\<\phi_{rt},\phi_{rtt}\>\\
		&\quad -\<\pa_t(\rho\psi\psi_r),\psi_t\>-\<\pa_t(P-\tilde{P})_r,\psi_t\>+\<\pa_t(h'(\tilde{\rho})\tilde{\rho}_r\phi),\psi_t\>\\
		&=:\sum_{i=1}^6 L_i.
	\end{align*}
	\noindent $\bullet$ (Estimate of $L_1$): Since $\rho_t =\phi_t = -\rho G-\rho_r\psi$, we obtain
	\begin{equation}\label{est:rhot}
		\|\rho_t\|_{L^\infty}\le \|\rho G\|_{L^\infty}+\|\rho_r\psi\|_{L^\infty}\le C(N+\|G_r\|),
	\end{equation}
	where we used $\|G\|_{L^\infty}\le C(\|G\|+\|G_r\|)$. Therefore, using $N<\e<1$,
	\[L_1\le C(N+\|G_r\|)\|\psi_t\|^2 \le C(1+\|G_r\|)\mathcal{F}.\]	
	 	
	\noindent $\bullet$ (Estimate of $L_2$ and $L_3$): Using integration by parts, we observe that
	\begin{equation}\label{est:L2}
		L_2=\kappa\<\rho_t A(\phi)_r,\psi_t\>+\kappa\<\rho A(\phi_t)_{r},\psi_t\>=\kappa\<\rho_t A(\phi)_r,\psi_t\>-\kappa\<A(\phi_t),\rho_r\psi_t\>-\kappa\<A(\phi_t),\rho G_t\>,
	\end{equation}
	where we used $\pa_r(\rho\psi_t r^{n-1})=r^{n-1}(\rho_r\psi_t+\rho G_t)$. Now, differentiating \eqref{eq:perturbation}$_1$ with respect to $t$, we obtain
	\[\rho G_t = -\phi_{tt} -\rho_t G -\rho_{rt}\psi -\rho_r\psi_t,\]
	which, after integration by parts, implies
	\begin{align*}
		-\kappa\<A(\phi_t),\rho G_t\> &= \kappa\<A(\phi_t),\phi_{tt}\>+\kappa\<A(\phi_t),\rho_t G+\rho_{rt}\psi+\rho_r\psi_t\>\\
		&=-\kappa\<\phi_{rt},\phi_{rtt}\>+\kappa\<A(\phi_t),\rho_tG+\rho_{rt}\psi+\rho_r\psi_t\>\\
		&=-L_3+\kappa\<A(\phi_t),\rho_tG+\rho_{rt}\psi+\rho_r\psi_t\>.
	\end{align*}	
	Substituting the above identity into \eqref{est:L2}, we get
	\begin{equation}\label{est:L2L3}
		L_2+L_3=\kappa\<\rho_t A(\phi)_r,\psi_t\> + \kappa\<A(\phi_t),\rho_tG+\rho_{rt}\psi\>.
	\end{equation}
	We use \eqref{est:Aphi_r} to get
	\[\kappa\<\rho_tA(\phi)_r,\psi_t\>\le C(N+\|G_r\|)(\|\psi_t\|+\|G_{r}\|+N)\|\psi_t\|\le C(\|G_r\|+\|G_r\|^2+N)(1+\mathcal{F}).\]
	Furthermore, it follows from the definition of $A(\phi_t)$ that
	\[\kappa\<A(\phi_t), \rho_tG +\rho_{rt}\psi\>=-\kappa\<\phi_{rt}, (\rho_tG +\rho_{rt}\psi)_r\>=-\kappa\<\phi_{rt},\phi_{rt}G + \phi_tG_r+\phi_{rrt}\psi+\phi_{rt}\psi_r\>.\]
	However, from
	\[\<\phi_{rt},\phi_{rrt}\psi\> = \frac{1}{2}\<\pa_r(\phi_{rt}^2),\psi\>=-\frac{1}{2}\<\phi_{rt}^2,G\>=-\frac{1}{2}\<\phi_{rt},\phi_{rt}G\>,\]
	and \eqref{est:G}, we get
	\[-\kappa\<\phi_{rt},\phi_{rt}G+\phi_{rrt}\psi\>=-\frac{\kappa}{2}\<\phi_{rt},\phi_{rt}G\>\le C\|G\|_{L^\infty}\|\phi_{rt}\|^2\le C(N+\|G_r\|)\mathcal{F}.\]
	Moreover, we use \eqref{est:rhot} to get
	\[-\kappa\<\phi_{rt},\phi_t G_r\>\le C\|\phi_t\|_{L^\infty}\|\phi_{rt}\|\|G_r\|\le C(N+\|G_r\|)\sqrt{\mathcal{F}}\|G_r\|\le C(\|G_r\|+\|G_r\|^2)(1+\mathcal{F}),\]
	and
	\[-\kappa\<\phi_{rt},\phi_{rt}\psi_r\>\le C\|\psi_r\|_{L^\infty}\|\phi_{rt}\|^2\le C(N+\|G_r\|)\mathcal{F},\]
	where we used $\|\psi_r\|_{L^\infty}\le C\|\psi_r\|^{1/2}\|\psi_{rr}\|^{1/2}\le C(\|\psi_r\|+\|\psi_{rr}\|)\le C(N+\|G_r\|)$. Combining all the estimates, and substitute into \eqref{est:L2L3}, we get
	\[L_2+L_3 \le C(1+\|G_r\|^2)(1+\mathcal{F}).\]
	
	\noindent $\bullet$ (Estimate of $L_4$): Using \eqref{est:rhot} and $\|\psi_r\|_{L^\infty}\le C(N+\|G_r\|)$, we estimate $L_4$ as
	\begin{align*}
		L_4 &= -\<\rho_t \psi\psi_r,\psi_t\>-\<\rho\psi_r,\psi_t^2\>-\<\rho\psi\psi_{rt},\psi_t\>\\
		&\le C\|\rho_t\|_{L^\infty}\|\psi\|_{L^\infty}\|\psi_r\|\|\psi_t\|+C\|\psi_r\|_{L^\infty}\|\psi_t\|^2+C\|\psi\|_{L^\infty}\|\psi_{rt}\|\|\psi_t\|\\
		&\le C(N+\|G_r\|)(1+\mathcal{F})+CN(\|G_t\|+\|\psi_t\|)\|\psi_t\|\le \frac{\mu}{2}\|G_t\|^2+C(1+\|G_r\|^2)(1+\mathcal{F}),
	\end{align*}
	where we utilized $G_t=\psi_{rt}+\frac{n-1}{r}\psi_t$.\\
	
	\noindent $\bullet$ (Estimate of $L_5$): Notice that
	\[\pa_t(P-\tilde{P})_r = \pa_r(P'(\rho)\phi_t)= P''(\rho)\phi_t\phi_r+P''(\rho)\phi_t\tilde{\rho}_r+P'(\rho)\phi_{rt}.\]
	Hence,  using \eqref{est:rhot}, we have
	\[L_5\le C(\|\phi_t\|_{L^\infty}\|\phi_r\|+\|\phi_t\|_{L^\infty}\|\tilde{\rho}_r\|_{L^2}+\|\phi_{rt}\|)\|\psi_t\|\le C(1+\|G_r\|^2)(1+\mathcal{F}).\]
	
	\noindent $\bullet$ (Estimate of $L_6$): Similarly, we notice that
	\[\pa_t(h'(\tilde{\rho})\tilde{\rho}_r\phi)=h'(\tilde{\rho})\tilde{\rho}_r\phi_t,\]
	and therefore,
	\[L_6\le C\|\phi_t\|_{L^\infty}\|\tilde{\rho}_r\|\|\psi_t\|\le C(1+\|G_r\|^2)(1+\mathcal{F}).\]
	Collecting all the estimates of $L_i$, we finally get
	\[\frac{d\mathcal{F}}{dt}+\frac{\mu}{2}\|G_t\|^2\le C(1+\|G_r\|^2)(1+\mathcal{F}),\]
	and the Gr\"onwall inequality implies that for all $t\in[0,T_*)$,
	\[\mathcal{F}(t)+1\le (\mathcal{F}(0)+1)\exp\left(\int_0^{T_*}C(1+\|G_r(s)\|^2)\,ds\right),\]
	which is finite since $T_*<+\infty$ and by the {\it a priori} estimate \eqref{est:apriori}.
\end{proof}
Therefore, Lemma \ref{lem:highest} shows that the highest derivative norm $\|(\phi_{rrr},\psi_{rr})\|$ is also bounded on $[0,T_*)$, which contradicts $\lim_{t\to T_*}\|(\phi,\psi)\|_{H^3\times H^2}=+\infty$. Thus, we conclude that $T_*=+\infty$ and the local solution constructed in Proposition \ref{prop:local} can be globally extended.

\subsection{Time-asymptotic behavior}

Now, we show the time-asymptotic behavior. Let $(\rho,u)$ be the global solution to \eqref{eq:NSK_spherical}. Our goal is to show that
\[\lim_{t\to\infty}(\|\phi(t)\|_{W^{1,\infty}}+\|\psi(t)\|_{L^\infty})=0.\]
First of all, due to the Sobolev inequality, 
\[\|\phi(t)\|^2_{L^\infty}\le C\|\phi(t)\|\|\phi_{r}(t)\|,\quad\|\psi(t)\|^2_{L^\infty}\le C\|\psi(t)\|\|\psi_r(t)\|.\]
Since $\|(\phi,\psi)(t)\|\le CN(t)<+\infty$, once we show that
\begin{equation}\label{time_goal}
	\lim_{t\to\infty}\|(\phi_r,\psi_r)(t)\|=0,
\end{equation}
we have
\[\lim_{t\to\infty}\|(\phi,\psi)(t)\|_{L^\infty}=0.\]
Furthermore, we again use Sobolev inequality to observe that
\[\|\phi_r(t)\|_{L^\infty}^2\le C\|\phi_r(t)\|\|\phi_{rr}(t)\|\le CN\|\phi_r(t)\|.\]
Therefore, \eqref{time_goal} also implies $\lim_{t\to\infty}\|\phi_r(t)\|_{L^\infty}=0$, which yields the desired convergence. Since $\|\psi_r\|\le \|G\|$, we will show that $\|\phi_r(t)\|,\|G(t)\|\to0$ as $t\to\infty$, which is guaranteed from $\|\phi_r(t)\|^2,\|G(t)\|^2\in W^{1,1}(0,\infty)$. Below, we prove this.\\

\noindent $\bullet$ ($\|\phi_r(t)\|^2,\|G(t)\|^2\in L^1(0,\infty)$) Since both terms $\|\phi_r\|^2$ and $\|G\|^2$ are included in the dissipation $\mathcal{D}$, it follows from the estimate \eqref{est:apriori} that 
\[\int_0^\infty (\|\phi_r(t)\|^2+\|G(t)\|^2)\,dt\le\int_0^\infty \mathcal{D}(t)\,dt<+\infty.\]

\noindent $\bullet$ ($\frac{d}{dt}\|\phi_r(t)\|^2,\frac{d}{dt}\|G(t)\|^2\in L^1(0,\infty)$): After taking time derivative and using $\phi_t = -\rho G-\rho_{r}\psi$,

\begin{align*}
	\frac{d}{dt}\|\phi_r\|^2 &= 2\<\phi_r,\phi_{tr}\>\le C\|\phi_r\|\|\phi_{tr}\|\\
	&\le C\|\phi_r\|^2 + C(\|\rho_rG\|^2+\|\rho G_r\|^2+\|\rho_{rr}\psi\|^2+\|\rho_r\psi_r\|^2)\\
	&\le C\mathcal{D},
\end{align*}
where we used
\[\|\rho_r\|_{L^\infty}\le \|\phi_r\|_{L^\infty}+\|\tilde{\rho}_r\|_{L^\infty}\le C(N+|\rho_b|)<C,\]
and Lemma \ref{lem:exp-Hardy} to obtain
\[\|\rho_{rr}\psi\|^2 \le C\|\phi_{rr}\psi\|^2+ C\|\tilde{\rho}_{rr}\psi\|^2\le C\|\phi_{rr}\|^2+C\|\psi_r\|^2.\]
This yields $\frac{d}{dt}\|\phi_r(t)\|^2\in L^1(0,\infty)$. For $G$, we use \eqref{eq:G} to obtain

\begin{align*}
	\frac{d}{dt}\|G\|^2 &= 2\<G,G_t\>=2\int_1^\infty G(\psi_t)_rr^{n-1}\,dr +2(n-1)\int_1^\infty G\psi_t r^{n-2}\,dr\\
	&=-2\int_1^\infty \psi_t \pa_r(G r^{n-1})\,dr+2(n-1)\int_1^\infty G\psi_t r^{n-2}\,dr\\
	&=-2\int_1^\infty \psi_t G_r r^{n-1}\,dr\le C\|\psi_t\|^2+C\|G_r\|^2<C\mathcal{D}.
\end{align*}

Hence, we also obtain $\frac{d}{dt}\|G\|^2\in L^1(0,\infty)$. This completes the proof of the time-asymptotic behavior and the proof of Theorem \ref{thm:main}.

\end{document}